\documentclass[12pt]{amsart}
\usepackage[utf8]{inputenc}

\usepackage[french,russian,english]{babel}
\usepackage{amsmath,amssymb,amsfonts,amsthm}
\usepackage[margin=1.3in]{geometry}
\usepackage[unicode]{hyperref}
\usepackage{graphicx}
\usepackage{todonotes}
\usepackage{tikz}
\usepackage[notref,notcite,final]{showkeys}
\usetikzlibrary{arrows}
\newcounter{thm}
\newtheorem{theorem}[thm]{Theorem}
\newtheorem*{theorem*}{Theorem}
\newtheorem{lemma}[thm]{Lemma}
\newtheorem{proposition}[thm]{Proposition}
\newtheorem*{conjecture*}{Conjecture}
\newtheorem{corollary}[thm]{Corollary}

\newtheorem{remark}[thm]{Remark}
\newtheorem*{remark*}{Remark}
\theoremstyle{definition}
\newtheorem{definition}[thm]{Definition}

\usepackage[backend=biber,firstinits,autolang=langname]{biblatex}
\renewcommand{\Im}{\mathop{\mathrm{Im}}}

\newcommand{\dist}{\mathop{\mathrm{dist}}}
\newcommand{\rot}{\mathop{\mathrm{rot}}}
\newcommand{\id}{\mathrm{id}}
\newcommand{\eps}{\varepsilon}
\newcommand{\bbC}{\mathbb C}
\newcommand{\bbR}{\mathbb R}
\newcommand{\bbQ}{\mathbb Q}
\newcommand{\cV}{\mathcal V}
\newcommand{\bbH}{\mathbb H}
\newcommand{\bbN}{\mathbb N}
\newcommand{\bbZ}{\mathbb Z}

\usepackage[hypcap=true]{caption}
\usepackage[hypcap=true]{subcaption}

\usepackage{color, soul}
\definecolor{gr}{rgb}{0.7, 1, 0.7}
\definecolor{rr}{rgb}{1, 0.7, 0.7}

\graphicspath{{figures/}}
\author{N. Goncharuk, I. Gorbovickis}
\title{Renormalization and scaling of bubbles}
\thanks{The research of the second author was supported by the German Research Foundation (DFG, project number  455038303).}

\begin{document}
 \maketitle

\begin{abstract}
 The paper explores scaling properties of bubbles --- a complex analogue of Arnold tongues, associated to a one-dimensional family of analytic circle diffeomorphisms.

 Bubbles are smooth loops in the upper half-plane attached at all rational points of the real line.
Results of a paper by X.~Buff and N.~Goncharuk (2015) show that the size of a $p/q$-bubble has order at most $q^{-2}$. 
In the current paper we improve this estimate by showing that the size of a $p/q$-bubble near a bounded-type irrational number $\alpha$ has order $d^{\xi(\alpha)} \cdot q^{-2}$, where $\xi(\alpha)>0$, and $d$ is the distance between $\alpha$ and $p/q$.

Proofs are based on a renormalization technique. In particular, $\xi(\alpha)$ is related to the unstable and the top stable eigenvalues of the renormalization operator at the rotation by $\alpha$.
\end{abstract}

\section{Introduction}

\subsection{Rotation numbers of circle maps and conjugacies to rotations}

Let  $f\colon \bbR/\bbZ \to \bbR/\bbZ$ be an orientation-preserving circle homeomorphism, and let $F\colon \bbR\to \bbR$ be its lift to the real line. Define
$$\rot F = \lim_{n\to +\infty}\frac{F^n (x)}{n};$$
then the rotation number of $f$, denoted by $\rot f\in\bbR/\bbZ$, is defined as
$$
\rot f = \rot F \mod 1.
$$
It is well known that the above limit always exists and the rotation number $\rot f$ is independent of the point $x$ and the choice of a lift $F$.
Furthermore, the rotation number $\rot f$ is invariant under topological conjugacies, and is rational if and only if $f$ has a periodic orbit. The following classical theorem due to J.-C.~Yoccoz (based on previous results by V.~Arnold and M.~Herman) shows that under certain conditions, analytic circle diffeomorphisms with the same rotation number are analytically conjugate.

For any $\alpha\in\bbR$, let $R_\alpha\colon\bbR/\bbZ\to\bbR/\bbZ$ be the rotation of the circle by~$\alpha$, i.e., $R_\alpha (z)=z+\alpha \mod 1$.

\begin{theorem}[J.-C. Yoccoz, \cite{Yoc}]
\label{th-Yoc}
There exists a full-measure set $\mathcal H\subset(\bbR\setminus \bbQ)/\bbZ$ such that if an analytic circle diffeomorphism $f$ has rotation number $\alpha \in \mathcal H$, then $f$ is analytically conjugate to $R_\alpha$: for some analytic circle diffeomorphism $h$, we have $f=h^{-1}\circ R_\alpha \circ h$.
\end{theorem}
The set $\mathcal H\subset (\bbR\setminus \bbQ)/\bbZ$ is called the set of Herman numbers. $\mathcal H$ is a proper subset of the set of Brjuno numbers $\mathcal B_C$ for any $C$ (see Definition~\ref{Brjuno_def}). It was shown by Yoccoz that the set of Herman numbers is optimal for the above theorem.

\subsection{Complex rotation numbers and bubbles}\label{complex_rot_number_intro_sec}
In \cite{Arn-english}, V.Arnold suggested the following construction of a \emph{complex rotation number} (the term is due to E.~Risler). Given an analytic orientation-preserving circle diffeomorphism $f\colon\bbR/\bbZ\to\bbR/\bbZ$ and a complex number $\omega$, $\Im \omega>0$, consider the quotient space of a neighborhood of the annulus $$A_\omega= \{z\in \bbC/\bbZ \mid 0\le \Im z \le \Im \omega\}$$ by the action of $f+\omega$, where $f+\omega$ is defined as
$$
(f+\omega)\colon z\mapsto f(z)+\omega\quad\mod\bbZ.
$$
This quotient space is a complex torus $T_{f+\omega}$. If we fix a lift $F$ of $f$ to the real line, this complex torus has naturally marked generators of the first homology group, namely  $\bbR/\bbZ$ and $[0, F(0)+\omega]$. Due to the Uniformization Theorem, there exists a uniquely defined complex number $\tau = \tau_F(\omega)\in \bbH$, called the \emph{modulus} of $T_{f+\omega}$, and a biholomorphism that takes this torus $T_{f+\omega}$ to the torus $T_\tau = \bbC/(\bbZ + \tau \bbZ)$, while taking the marked generators to the generators $\bbR/\bbZ$ and $\tau\bbR/\tau\bbZ$  of the torus $T_\tau$.

Consider the map $\tau_F\colon \omega\mapsto \tau_F(\omega)$, $\tau_F\colon \bbH\to \bbH$.
In \cite{Arn-english}, V.~Arnold posed a conjecture on a limit behavior of $\tau_F$ near the real axis: he conjectured that if $\rot f$ is Diophantine, then 
$$
\lim_{\eps\to 0} \tau_F(i\eps) = \rot F.
$$
Construction of the complex rotation number and Arnold's conjecture were strongly motivated by his work on analytic conjugacies of analytic circle maps to rotations. Namely, the uniformizing map of the torus $T_{f+\omega}$ conjugates $f+\omega$ to the translation $z\to z+\tau_F(\omega)$. If we assume that we can pass to a limit as $\omega=i\eps\to 0$, the limit map will conjugate the circle map $f$ to the translation by $\lim_{\eps\to 0}\tau_F(i\eps)$. 
 
The conjecture was proved by E.~Risler \cite{Ris} and V.~Moldavskij \cite{M-en} independently. 
This result justifies the term ``complex rotation number'' for $\tau_F$. 
 The limit behavior of $\tau_F$ on the real axis for non-Diophantine rotation numbers was further studied in a sequence of papers  \cite{YuIM}, \cite{NG2012-en}, \cite{XB_NG}, \cite{NGinters}, \cite{NGselfsim}. In particular, we have the following result:
\begin{theorem}[\cite{XB_NG}]\label{Cont_ext_theorem}
For any orientation-preserving analytic circle diffeomorphism $f\colon\bbR/\bbZ\to\bbR/\bbZ$ and its lift $F\colon\bbR\to\bbR$, the map $\tau_F$ is analytic in the upper half-plane and extends continuously to the real axis.
\end{theorem}
We will use the same symbol $\tau_F$ for this continuous extension. That is, we will view $\tau_F$ as a map $\tau_F\colon\overline{\bbH}\to\overline{\bbH}$.
\begin{definition}
 The \emph{complex rotation number} of the lift $F$ of  a circle diffeomorphism $f$ is the limit
$$
\tau(F) = \lim_{\substack{\omega\in\bbH\\ \omega\to 0}}\tau_F(\omega) = \tau_F(0).
$$
\end{definition}

The boundary behavior of the complex rotation number is further described by the following theorem. 
Recall that a circle diffeomorphism is called \emph{hyperbolic} if it has a rational rotation number and the multipliers of all its periodic orbits are not equal to $\pm 1$. 

\begin{theorem}[\cite{XB_NG}]\label{bubble_theorem}
Let $f$ and $F$ be the same as in Theorem~\ref{Cont_ext_theorem}. The equality 
$$
\tau(F) = \rot F
$$
holds if and only if $f$ is not a hyperbolic diffeomorphism. Furthermore, if $f$ is hyperbolic and $\rot F = p/q$, then $\Im(\tau(F))>0$ and $\tau(F)$ is located within the disc of radius $q^{-2} \times D_f/4\pi$ tangent to $\bbR$ at $p/q$, where $D_f = \int_{S^1} |f''/f'| dx$ is the distortion of $f$.
\end{theorem}
It is insightful to consider behavior of the complex rotation number in families of analytic maps, similarly to the classical construction of Arnold's tongues for circle diffeomorphisms. 
\begin{definition}
Let $I\subset\bbR$ be a closed interval. We say that a one-parameter family of analytic orientation preserving circle diffeomorphisms $\mathcal F=\{f_t\}_{t\in I}$ is \emph{monotonic} if 
\begin{itemize}
	\item for any $x\in\bbR/\bbZ$ and $t\in I$, $f_t(x)$ depends smoothly on $t$, and $\frac{\partial f_t}{\partial t}(x)>0$, with one-sided derivatives taken at the endpoints of~$I$;

	\item there exists $\eps>0$, such that $f_t\in\mathcal D_\eps$ for any $t\in I$.
\end{itemize}
\end{definition}
Given a monotonic family of circle diffeomorphisms $\mathcal F=\{f_t\}_{t\in I}$, we let $\hat {\mathcal F}=\{F_t\}_{t\in I}$ denote a family of lifts of $f_t$ to the maps of the real line, so that $F_t(x)$ depends continuously on $t$, for any $t\in I$ and $x\in\bbR$. For any rational number $p/q$, define the p/q-\emph{mode-locked interval} by 
$$
I_{p/q, \hat{\mathcal F}} = \{t\in I \mid \rot(F_t)=p/q\}.
$$
It is obvious from monotonicity of $\mathcal F$ that $t\mapsto \rot F_t$ is a weakly increasing function, so $I_{p/q, \hat{\mathcal F}}$ is either an empty set, or a closed interval that possibly degenerates into a single point. (The latter happens only if $F_t^q - p = id$ for some $t\in I$.)

The boundary behavior of a map $t\to \tau(F_t)$ on the real axis produces a fractal-like set ``Bubbles''. 
\begin{definition}
For a monotonic family $\mathcal F=\{f_t\}_{t\in I}$ of analytic circle diffeomorphisms and a corresponding family of lifts $\hat{\mathcal F}=\{F_t\}_{t\in I}$, a \emph{$p/q$-bubble} $B_{p/q, \hat {\mathcal F}}$ of $\hat {\mathcal F}$ is
 the image under the map $t\mapsto \tau(F_t)$ of the set $I_{p/q, \hat {\mathcal F}}$.
 
 The set of all bubbles of $\mathcal F$ is defined as
$$
B(\mathcal F):=\bigcup_{p/q\in[0,1]\cap\bbQ} B_{p/q,\hat{\mathcal F}}.
$$

\end{definition}

Due to Theorem~\ref{bubble_theorem}, if  the set $I_{p/q, \hat{\mathcal F}}$ is nonempty and is contained in the interior of $I$, then the $p/q$-bubble of the family $\hat{\mathcal F}$ is the union of one or several curves in $\bbH$ that start and end at $p/q$. Each of these curves corresponds to an interval of hyperbolicity of $f_t$.  
Due to Proposition~\ref{crot_analytic_dependence_prop} below, the curves of the bubble are at least as smooth as the family $\hat{\mathcal F}$.

 It is natural to ask about fractal properties of the set $B(\mathcal F)$, in particular, about the asymptotic size of the bubble attached at $p/q$. This resembles questions on the size and scaling of the limbs of the Mandelbrot set (see \cite{Kapiamba} and references therein). 
 
Formally, let $D_{\eps, r}\subset \bbH$ be the disc of diameter $\eps$ that is tangent to $\bbR$ at $r$. Let the (hyperbolic) \emph{size} of a $p/q$-bubble be the diameter of the smallest possible disc  $D_{\eps, p/q}$ that contains this bubble.  The paper addresses the following question. 

\vskip 0.2 cm

\textbf{Question: what is the asymptotic size of the $p/q$-bubbles, as $p/q$ approaches a particular value?}

\vskip 0.2 cm

 Due to Theorem~\ref{bubble_theorem}, the size of a $p/q$-bubble is at most $C q^{-2}$, where the constant $C>0$ depends on the family $\mathcal F$. 

In Section \ref{sec-rational}, we show that this estimate is sharp when $p/q$ approaches a rational number. %near rational numbers. 
This is an easy corollary of the previous result \cite{NGselfsim} on self-similarity of $B(\mathcal F)$ near rational numbers.

The main result of the current paper shows that  this estimate  is not sharp when $p/q$ approaches an irrational number of bounded type. %near bounded-type irrational numbers.  
(The numbers of bounded type are defined in Sec.~\ref{sec-arithm}.)

\begin{theorem}[Main Theorem]
Let $\alpha\in \bbR\setminus \bbQ$ have a type bounded by $k$. 
Take any monotonic family of analytic circle diffeomorphisms $\mathcal F = \{f_t\}_{t\in I}$, let $\hat{\mathcal F}=\{F_t\}_{t\in I}$ be the family of lifts.  Assume that $t_0\in I$ is an interior point of $I$, such that $\rot (F_{t_0})=\alpha$.

Then any $p/q$-bubble of the family $f_t$ is   located in a disc of radius $$c\cdot  \dist (\alpha, p/q)^{\xi} \cdot q^{-2}$$ tangent to $\bbR$ at $p/q$, where $c = c(\alpha,\mathcal F)>0$ depends on $\alpha$ and the family $\mathcal F$, and $\xi>0$ depends only on $k$.
\end{theorem}
The following result estimates $\xi$ in a particular case.
\begin{proposition}\label{prop-gold}
 For a rotation number $\alpha=\phi = \frac{\sqrt{5}-1}{2} = [1, 1, 1, \dots]$, we have $\xi>1$.
\end{proposition}
 A stronger form of the Main Theorem appears in Sec.~\ref{sec-mainth}.

The proof of the Main Theorem relies on renormalization. Renormalizations of circle diffeomorphisms were widely used starting from the works of Yoccoz on conjugacies to rotations. However, the result on the hyperbolicity of the renormalization operator for circle diffeomorphisms is recent \cite{GY}. In fact, the first author got interested in renormalization theory (which eventually led to the paper \cite{GY}) after realizing, in discussions with the second author, that hyperbolicity of renormalization is needed to obtain the scaling of bubbles.

 The question on the sizes of bubbles near irrational numbers of unbounded type (Liouville points in particular) is widely open.

Another interesting open question is whether complex rotation numbers generalize to critical analytic circle maps and whether the known results on hyperbolicity of renormalization imply results on self-similarity of bubbles for critical circle maps.

\section{Preliminaries}
\subsection{Arithmetics}
\label{sec-arithm}
Let $G(x) = \{1/x\}$ be the Gauss map defined for all $x>0$, where the curly brackets denote the fractional part of the number. Given a real number $\alpha \in \bbR$, define the numbers $\alpha_{-1},\alpha_0,\alpha_1,\ldots\in\bbR$ inductively by
$$
\alpha_{-1}=1,\, \alpha_0=\alpha, \,\,\alpha_n=G(\alpha_{n-1})\text{ for }n\ge 1.
$$
Put $k_n=[1/\alpha_{n}]$, $n\geq 0$. For rational numbers $\alpha$, this sequence is finite. The numbers $k_n$ are the coefficients of the (finite or infinite) continued fraction expansion of $\alpha$:
$$
\alpha = \cfrac{1}{k_0 + \cfrac{1}{k_1+\dots}}.
$$
 We will abbreviate the notation for the continued fractions by writing
 $$\alpha=[k_0,k_1,k_2,\ldots].$$
Let
$$\frac{p_n}{q_n}=[k_0,\ldots,k_{n-1}].$$
The number $p_n/q_n$ is called the $n$-th convergent to $\alpha$. For irrational numbers $\alpha$, we have $\alpha = \lim_{n\to \infty} p_n/q_n$.   If $\alpha$ is rational, this sequence is finite and the last term $p_N/q_N$ coincides with $\alpha$.  In both cases, we will write $\alpha \sim \{p_n/q_n\}$ to indicate that the sequence of rational numbers $p_n/q_n$ is the sequence of the $n$-th convergents to $\alpha$. %for the correspondence between the number $\alpha$ and its continued fractional convergents.
It is also convenient to define 
$$
p_0:=0\qquad\text{and}\qquad q_0:= 1.
$$

An irrational number $\alpha$ is called a \textit{number of bounded type} if the coefficients $k_n$ of its continued fraction expansion  are bounded (i.e., $\sup_{n\in\bbN} k_n<\infty$). More specifically, given a constant $k\ge 1$, an irrational number $\alpha$ is \textit{of type bounded by $k$}, if $\sup_{n\in\bbN} k_n\le k$. All numbers of bounded type are Herman numbers. 

\subsection{Complex rotation numbers: invariance and dependence on parameters}

Here we collect several useful properties of complex rotation numbers. 

\begin{lemma}\label{analyt_conjugacy_lemma}
For any two analytically conjugate circle diffeomorphisms $f_1$ and $f_2$, and their corresponding analytically conjugate lifts $F_1$ and $F_2$, we have $\tau(F_1)=\tau(F_2)$.
\end{lemma}
For the proof of Lemma~\ref{analyt_conjugacy_lemma}, see~\cite[Lemma~8]{NGinters}.

For $\eps>0$, let $\Pi_\eps = \{z\in \bbC/\bbZ \mid -\eps<\Im z < \eps\}$ be an equatorial annulus on the cylinder $\bbC/\bbZ$.
Let  $\mathcal D_{\eps}$ be the set of all bounded analytic maps $f\colon \Pi_\eps\mapsto\bbC/\bbZ$ that are defined in $\Pi_\eps$ and extend continuously to the boundary. In Section~\ref{Banach_space_subsec} we will see that $\mathcal D_{\eps}$ has a structure of an affine complex Banach manifold. 
We will need the following proposition on the analytic dependence of $\tau_F$ on $F$.

\begin{proposition}\label{crot_analytic_dependence_prop}
Let $f$ and $F$ be the same as in Theorem~\ref{Cont_ext_theorem}, and let $\eps>0$ be such that $f\in \mathcal D_\eps$. Fix $\omega$ so that
\begin{itemize}
 \item either $\omega\in\bbH$,
 \item or $\omega\in\bbR$ and $f+\omega$ is a hyperbolic circle diffeomorphism.
\end{itemize}
Then the correspondence $(\omega, f)\mapsto \tau_F(\omega)$ (with the choice of $F$ that depends continuously on $f$) extends to a complex analytic map $T$ on some sufficiently small neighborhood $D\times \mathcal U\subset \bbC \times \mathcal D_\eps$ of $(\omega,f)$.
\end{proposition}

This proposition was essentially proved in \cite[Chapter 2, Proposition 2]{Ris}. Namely, Proposition 2 from \cite{Ris} states that if a complex torus is glued from the annulus in $\bbC/\bbZ$, and the gluing depends analytically on the parameters, then the  modulus of the resulting complex torus also depends analytically on the parameters. The proof is based on the Ahlfors-Bers theorem on existence and analytic dependence of solutions of Beltrami equations with respect to  additional parameters.  This immediately implies the case $\omega\in \bbH$, and the reduction for hyperbolic circle diffeomorphisms is contained in \cite{NG2012-en}.

\section{Scaling of  bubbles near rational points}
\label{sec-rational}

In \cite{NGselfsim}, it was proved that for a generic monotonic family $\mathcal F =\{f_t\}_{t\in I}$ of analytic circle diffeomorphisms and its lifts $\hat{\mathcal F}=\{F_t\}_{t\in I}$, the set of all bubbles $B(\mathcal F)$ has a limiting shape near any rational point $k/l$, when viewed in the appropriate (M\"obius) chart. In this section, we formulate this result for the right semi-neighborhood of $k/l=0$, and explain how this implies that the $p/q$-bubble has the size $\sim cq^{-2}$ when $p/q\to k/l$, where $c$ depends on $k/l$ and the family $\mathcal F$.

Suppose that the map $f_0$ in the family $\mathcal F=\{f_t\}$ has a single quadratic parabolic fixed point at $0$: $f_0(0)=0$, $f_0'(0)=1$, $f_0''(0)>0$. Note that monotonicity of the family $\mathcal F$ implies that for arbitrarily small $t>0$, the parabolic fixed point at $0$ splits into two distinct complex conjugate fixed points of $f_t$. Let $\Psi^-$, $\Psi^+$ be Fatou coordinates for $f_0$ in the right and left semi-neighborhoods of $0$ respectively. Note that $\Psi^-$ extends to a neighborhood containing the interval $(-1, 0) $ and $\Psi^+$ extends to a neighborhood containing the interval $(0,1)$  via the dynamics of $f_0$. Let $\mathbf K = \Psi^{-}((\Psi^+)^{-1}(z)-1)$ be the transition map between the Fatou coordinates. Note that $\mathbf K$ is a circle map. The family of circle maps $\mathcal K = \{x\mapsto \mathbf K(x)+a\}_{a\in [0,1]}$ coincides with a well-known family of Lavaurs maps (``maps through the eggbeater'') written in the chart $\Psi^-$.  Generically, $\mathbf K$ is not a rotation. 

 Let $G_k(z)=1/z-k$; this is the analytic extension of the corresponding branch of the Gauss map $G\colon x\mapsto \{1/x\}$.
 
\begin{theorem}[\cite{NGselfsim}]
\label{th-selfsim}
 For a generic monotonic family $\mathcal F = \{f_t\}_{t\in I}$ of circle diffeomorphisms as above, 
  the countable union of analytic curves $$-G_n\left(\bigcup_{ \frac ab\in  [\frac 1n, \frac 1{n+1}]\cap\bbQ} B_{a/b, \hat{\mathcal F}}\right) $$ tends uniformly to the set $B(\mathcal K)$ of all bubbles of the family $\mathcal K = \{x\mapsto \mathbf K(x)+a\}_{a\in [0,1]}$ as $n\to \infty$.
  
  In particular, for any $p/q\in [0,1]\cap \bbQ$, the curves $-G_n( B_{q/(nq+p),  f_{t}}) $ tend uniformly to the $p/q$-bubble of the family $\mathbf K+a$ as $n\to \infty$. 
\end{theorem}
The proof is based on the near-parabolic renormalization for the circle map $f_t$, $t\approx 0$. Similarly to Theorem \ref{th-renorm-bubbles} proved below, application of the renormalization acts as $-G_n$ on the complex rotation number (the discrepancy with Theorem \ref{th-renorm-bubbles} is due to a different choice of orientation when rescaling the fundamental domain). Since near-parabolic renormalizations of $f_t$ tend to the family of Lavaurs maps $\mathbf K+a$ as $t\to 0$, the complex rotation numbers in the renormalized family tend to the complex rotation numbers for  $\mathbf K+a$.

An analogous statement holds in a neighborhood of each rational point $k/l$; the map $z\to -G_n(z)$ in this case should be replaced by a M\"obius map from $PSl(2, \bbZ)$ that takes $k/l$ to infinity.

\begin{corollary}
\label{cor-geom}
1. In the assumptions of Theorem \ref{th-selfsim}, there exists a constant $c = c(\mathcal F)$ such that the $1/n$-bubble of the family $\mathcal F = \{f_t\}$ has size bounded below by $c n^{-2}$. 

%Moreover, for any $p/q$, the $1/(n+p/q) = \frac{q}{qn+p}$-bubble has size bounded above and below by $c/(qn+p)^2$.

%2. There exist two  disks $D_1\subset D_2$ in $\bbH$ tangent to $\bbR$ at $0$ such that the set $B(\mathcal F)\setminus B_{0,}$ of all bubbles except the one at zero has no common points with $D_1$

2. There exist two  disks $D_1\subset D_2$ in $\bbH$ tangent to $\bbR$ at $0$ such that no $p/q$-bubble of the family $\mathcal F$ intersects   $D_1$ for $p/q\neq 0$, and there are infinitely many bubbles of $\mathcal F$ that intersect $D_2$.
\end{corollary}
A similar statement holds near any rational point $k/l$ instead of $0$.

Note that Theorem \ref{bubble_theorem} implies the upper bound on the size of the bubble: the $1/n$-bubble has size at most $C/n^2$ for $C=C(\mathcal F)$. The first part of the Corollary means that the asymptotic size of $p/q$-bubbles near any rational point has order exactly  $c q^{-2}$.

\begin{proof}
1.
The previous theorem implies that the images of the bubbles $B_{1/n, \hat{\mathcal F}}$ under the maps  $-G_n$ tend to the $0$-bubble of the family $\mathcal K = \{\mathbf K+a\}$. Since generically, $\mathbf K$ is not a rotation, this bubble is non-degenerate and thus has a nonzero size $c$. 
%Similarly, the images of the bubbles $B_{q/(qn+p), ft}$ tend to the $p/q$-bubble of $\mathbf K+c$. 
Since the image of a disc $D_{c n^{-2}, 1/n}$ under $(-G_n)$  is   $D_{c, 0}$, the estimate follows.

2. Due to Theorem \ref{bubble_theorem}, the size of a $p/q$-bubble is at most $C q^{-2}$ where $C$ only depends on the family $f_t$. 
It is easy to check that the discs $D_{C q^{-2}, p/q}$ do not intersect the disc $D_1= D_{1/(2C), 0}$, which implies the first part of the statement.

As mentioned above, generically, the family $\mathbf K +a$ has a non-degenerate zero bubble. Choose $c$ so that this bubble intersects a half-plane $\{\Im z\ge c\}$. 
Then the images of the bubbles $B_{1/n, \hat{\mathcal F}}$ under $(-G_n)$ intersect a half-plane $\{\Im z\ge c/2\}$ for sufficiently large $n$, thus the bubbles $B_{1/n, \hat{\mathcal F}}$ intersect the disc $D_2 = \{\Im (-1/z) \ge c/2 \}$ for large $n$. This completes the proof of the second statement.

\end{proof}

\section{Refined version of the Main theorem}
\label{sec-mainth}

\begin{theorem}
\label{th-main}
	For any $k>0$, there exists a positive constant $\Lambda=\Lambda(k)<1$ with the following property.
	
	Take any monotonic family of analytic circle diffeomorphisms $\mathcal F = \{f_t\}_{t\in I}$ and a corresponding family of its lifts $\hat{\mathcal F}=\{F_t\}_{t\in I}$ that depends continuously on $t$. Fix any $\alpha$ of type bounded by $k$, and let $p_n/q_n$ be the sequence of its continued fractional convergents. Assume that $t_0\in I$ is an interior point of $I$, such that $\rot (F_{t_0})=\alpha$.
 
 Then there exists a constant $c=c(\alpha, \mathcal F)$ such that for any integer $r\ge 1$,
	every $p/q$-bubble of the family $f_t$ with
	$p/q\in [p_{r}/q_{r}, p_{r+1}/q_{r+1}]$ is located in a disc of radius $ c\Lambda^r q^{-2}$ tangent to $\bbR$ at $p/q$.
% 	% for the renormalization $g_a = R_n (f+a)$ in this segment of parameters, all renormalized maps $g_a$ will be $\lambda^n$-close to rotations in $C^2$ metric.
\end{theorem}
As we will see below, $\Lambda$ is related to the two top eigenvalues of the renormalization operator over the orbit of the rigid circle rotation $R_\alpha$ under the renormalization. 
%In particular, numerical experiments \cite{} show that for a golden ratio, $\rot f= \frac{\sqrt{5}-1}{2}$, we have  $\lambda < 0.618$. 

The above theorem implies the Main Theorem.

\begin{proof}[Reduction of the Main Theorem to Theorem \ref{th-main}]

 Using recurrent relations on the continued fractional convergents, we get $p_{n} = k_{n-1}p_{n-1}+p_{n-2} \le  kp_{n-1}+p_{n-2}$, $ q_{n}= k_{n-1}q_{n-1}+q_{n-2} \le  kq_{n-1}+q_{n-2}$, and thus $p_n,q_n$ grow at most exponentially fast: $q_n <  (k+1)^{n}$. 

We will use Theorem \ref{th-main} for the interval $[\frac{p_{r}}{q_{r}}, \frac{p_{r+2}}{q_{r+2}}]\subset [\frac{p_{r}}{q_{r}}, \frac{p_{r+1}}{q_{r+1}}]$. 
The distance $d=\dist(\alpha, p/q)$ for $$p/q\in \left[\frac{p_{r}}{q_{r}}, \frac{p_{r+2}}{q_{r+2}}\right]$$ is at least $|p_{r+2}/q_{r+2} - \alpha|$. Further, we can prove by induction that $$|p_{r+2}-\alpha q_{r+2}| = \alpha_0 \alpha_1 \dots \alpha_{r+1}  >( k+1)^{-r-2}, $$ thus $d>(k+1)^{-r-2}  (q_{r+2})^{-1}> (k+1)^{-2r-4}.$ We conclude that $r > -0.5\log_{k+1} d-2$.

According to the previous theorem, the size of a $p/q$-bubble that is rooted on a segment $ [p_{r+2}/q_{r+2}, p_{r}/q_{r}]$ is at most $$c(\mathcal F)(\Lambda(k))^r q^{-2} <  c(\mathcal F) (\Lambda(k))^{- 0.5\log_{k+1} d } q^{-2} = c(\mathcal F) d^\xi q^{-2}$$
where $\xi = -0.5\log_{k+1}\Lambda(k)$. We have $\xi >0  $ since $\Lambda(k)<1$. \end{proof}

Let us also derive Proposition \ref{prop-gold} on the value of $\xi$ for the golden ratio rotation number $\alpha=\phi=\frac{\sqrt{5}-1}{2}$. 
\begin{proof}[Proof of Proposition \ref{prop-gold}]
For this value of $\alpha$, the estimates from the previous proof can be improved as follows. Since $p_n, q_n$ are Fibonacci numbers, we have $q_n > c_1\phi^{-n}$, for some constant $c_1>0$, and $|p_{n}-\alpha q_n|=\phi^{n}$. Thus, using the same notation as above, we obtain $d > \frac{1}{c_1}\phi^{2r+4}$. Same computations as above, but with the improved inequality on $d$, imply $\xi = \log \Lambda / \log \phi^{2}$.

The value of $\Lambda$ for the golden ratio is estimated in Proposition \ref{prop-golden} below: $\Lambda<\phi^2$. Since both logarithms are negative, this implies $\xi>1$.  

\end{proof}

In Corollary \ref{cor-geom}, we obtained the geometric interpretation for the scaling of bubbles near rational points. 
We observe that Theorem \ref{th-main} implies quite a different behavior near bounded-type irrational points. 
\begin{corollary}
	Let an irrational number $\alpha$ and a family of circle diffeomorphisms $\mathcal F$ be the same as in Theorem~\ref{th-main}. Then for any disc $D_{\eps, \alpha}\subset\bbH$ tangent to the real line at $\alpha$, all $p/q$-bubbles of the family $\mathcal F$ for $|\alpha-p/q|$ sufficiently small do not intersect the disc $D_{\eps, \alpha}$.
	
%In the assumptions of Theorem \ref{th-main}, for any disc  $D_{\eps, \alpha}$, the $p/q$-bubbles for $|\alpha-p/q|$ sufficiently small, $p/q\neq \alpha$,  do not intersect this disc.
\end{corollary}
\begin{proof}
It is easy to check that  for $\delta=1/2\eps$, the disc  $D_{\eps, \alpha}$ does not intersect the discs $D_{\delta q^{-2} ,p/q}$ for all $p/q\in \bbQ$, $p/q\neq 0$. Let $c=c(\alpha,\mathcal F)$ be the same as in Theorem \ref{th-main}. Choosing $N$ such that  $c\Lambda^N<\delta$, we conclude that due to Theorem \ref{th-main}, the disc $D_{\eps, \alpha}$ will not intersect any bubbles that grow on the interval  $[p_{N}/q_N, p_{N+1}/q_{N+1}]$ containing $\alpha$. 
\end{proof}

\section{Renormalization for circle diffeomorphisms}

\subsection{Yoccoz's renormalization}

The proof of Theorem \ref{th-main} heavily uses renormalization of circle diffeomorphisms.

In this section we will always identify the circle with the affine manifold $\bbR/\bbZ$. For any two points $a,b\in\bbR/\bbZ$ that are not antipodal, let $[a,b]=[b,a]\subset\bbR/\bbZ$ denote the shortest arc connecting these two points. Similarly, if $a,b\in\bbC$, then the line segment connecting these two points will be denoted by $[a,b]=[b,a]$.

Given a circle diffeomorphism $f$, the renormalization $\mathcal R f$ is defined as a rescaled first-return map of $f$ to a fundamental arc $I=[0, f^{q}(0)]$, where $q$ is a closest return time of zero for $f$.
Different versions of renormalization differ by the choice of $q$ and the rescaling. The latter one, in particular, does not have to be affine. 
In the proof of Theorem \ref{th-Yoc}, Yoccoz \cite{Yoc} introduced the following (nonlinear) analytic rescaling. Let $q_n$ be the denominator of the $n$-th convergent $p_n/q_n$ of $\rot f$, written as an irreducible fraction.  Let $U$ be an open neighborhood of $I=[0, f^{q_n}(0)]$ in $\bbC/\bbZ$, and consider an analytic real-symmetric map $\phi\colon U/f^{q_n}\to\bbC/\bbZ$, defined on the annulus $U/f^{q_n}$ and taking it diffeomorphically onto its image. The lift of this map $\phi$ to the map from $U$ to $\bbC/\bbZ$ is used as a rescaling coordinate in the definition of $\mathcal Rf$.  With this construction, the renormalization of an analytic circle diffeomorphism is again an analytic circle diffeomorphism.

A similar construction was used by E.~Risler in \cite{Ris}. 
In \cite{GY}, the first author in collaboration with M.~Yampolsky proved the hyperbolicity result for the renormalization operator of this type. %We restricted the renormalized map to some fixed strip around the real line to obtain a self-map of a certain affine Banach space. 
Below we describe the construction for the renormalization $\mathcal R$ used in \cite{GY}.
The constructions  from \cite{Ris} and \cite{GY} work in a more general case: for analytic maps $f\colon \bbR/\bbZ \to \bbC/\bbZ$ close to rotation, but not necessarily preserving the circle $\bbR/\bbZ$. We will only use this construction for maps that preserve $\bbR/\bbZ$.

\subsection{Definition of the renormalization operator}

\subsubsection{Banach space: domain of renormalization}\label{Banach_space_subsec}
We recall that for any $h>0$, the set $\Pi_h = \{z\in \bbC/\bbZ \mid -h<\Im z < h\}$ is the equatorial annulus of width $2h$ on the cylinder $\bbC/\bbZ$. The functional space $\mathcal D_{h}$ consists of all bounded analytic maps $f\colon \Pi_h\mapsto\bbC/\bbZ$ that are defined in $\Pi_h$, and extend continuously to the boundary.

Let $\tilde\Pi_h = \{z\in \bbC \mid -h<\Im z < h\}$ be the strip around the real axis. The space $\mathcal D_h$ can be equipped with an affine complex Banach manifold structure, modeled on the Banach space $\tilde{\mathcal D}_h$ of all $1$-periodic bounded analytic maps $G\colon \tilde\Pi_h\to\bbC$ that are defined in $\tilde \Pi_h$, and extend continuously to the boundary. (The space $\tilde{\mathcal D}_h$ is a complex Banach manifold with respect to the sup norm in $\tilde\Pi_h$.) The atlas is constructed as follows: if $F\colon \tilde\Pi_h\to\bbC$ is a lift of $f\in\mathcal D_h$, then the correspondence $f\mapsto (F -\mathrm{id})$ provides local charts on $\mathcal D_h$, assuming that $F$ depends continuously on $f$. The transition maps between such charts are affine.

\subsubsection{Fundamental domain and the function $n(f)$}

First, choose the fundamental domain for renormalization.
  For a number $\alpha\in \bbR/\bbZ$, let $p_n/q_n$ be its continued fractional convergents. Find the smallest number $m$ such that $0<q_m\alpha-p_m<0.01.$
We define $n(\alpha):=q_m$. It was shown in \cite{GY} that $n(\alpha)$ is well defined and locally constant everywhere outside a certain closed countable subset $K$ that consists of rational numbers and whose only accumulation points are the rational points of the form $p/q$ with $q<100$.

Consider the set  
$$
\mathcal T=\{R_\alpha \mid \alpha\in (\bbR/\bbZ)\setminus K\}
$$
of all rigid rotations by the angles $\alpha\in (\bbR/\bbZ)\setminus K$.

  Let $\mathcal U_h\subset \mathcal D_h$ be the union of disjoint neighborhoods of connected components of $\mathcal T$; we assume that all maps  $f\in \mathcal U_h$ are univalent. Then $n(\cdot)$ extends as a continuous locally constant function on $\mathcal U_h$.
  For $f\in \mathcal U_h$, we write $n=n(f)$ and use a fundamental domain of $f^n$ in the renormalization construction for $f$ below.

\subsubsection{Renormalization operator}\label{renorm_subsection}

Fix any $h\in \bbR^+$ and take $f\in \mathcal U_h$. Put $f^n(0)-0 = :L$ where $n=n(f)$ is defined above. Consider the segment $I$ joining the points $-2ih L$ and $2ih L$. Let $R$ be a curvilinear quadrilateral bounded by the segment $I$, its image $f^n(I)$, and two straight line segments joining their endpoints.

We assume that the domain $\mathcal U_h$ is sufficiently small so that any $f\in \mathcal U_{h}$ is sufficiently close to a rotation. Then these four curves are simple and bound a domain in $\Pi_{h}$.

There exists a conformal map $\Psi \colon R \to \bbC$, $\Psi(0)=0$, that extends conformally to the union of $R$ and a neighborhood of the interval $I$, where it conjugates $f^n$ to the shift by $(-1)$.   
The map $\Psi$ descends to the map $\tilde \Psi\colon R/f^n \mapsto \bbC/\bbZ$. The conditions above do not define the map $\Psi$ uniquely. The exact choice of $\Psi$ is described in \cite{GY}.  In particular, $\Psi$ is chosen so that if $f$ preserves the real axis, then $\Psi$ also preserves the real axis, and if $f$ is a rigid rotation, then $\Psi$ is affine.

Let $P\colon R\to R$ be the (partially defined) first-return map to $R$ under the iterates of $f$.

\begin{definition}
For any $h>0$ and $f\in\mathcal U_h$, the \textbf{renormalization} $\mathcal R_h f$ of $f$ is defined as the composition
$$
\mathcal R_h f := \tilde \Psi \circ P \circ \tilde \Psi^{-1}.
$$
\end{definition}
Note that the map $\tilde \Psi$ depends not just on $f$, but also on the parameter $h$, so the renormalizations $\mathcal R_h f$ and $\mathcal R_{\tilde h} f$ do not necessarily coincide when $h\neq \tilde h$.

It follows from the definition that if $f\in\mathcal U_h$ preserves the real circle (i.e., if $f$ is a circle diffeomorphism), then $\mathcal R_h f$ is a circle diffeomorphism as well and
$$
\rot (\mathcal R_h f) = G^{m+1}(\rot f)=\frac{-(\rot f) q_{m+1}+p_{m+1}}{(\rot f) q_m - p_m},
$$
where $\rot f$ is assumed to be in the interval $[0,1)$ with $\rot f\sim \{p_k/q_k\}$, and the index $m\ge 0$ is such that $n(f)=q_m$. 

%\begin{remark}
%The maps $\Psi$, constructed in~\cite{GY}, satisfy the following condition: if $f$ is a rigid rotation by the angle $\alpha\in(\bbR/\bbZ)\setminus K$, then the map $\Psi$ is linear, and thus $\mathcal R_h f$ is again a rotation. 
%\end{remark}

Note that if $f$ is a rotation, then the Riemann surface $R/f^n$ is conformally equivalent to the cylinder $\Pi_{2hL/L} = \Pi_{2h}$ and $\Psi$ is a linear expansion in $1/L$ times, thus  $\mathcal R_h f$ is again a rotation. Now, for maps $f$ sufficiently close to rotations,  it follows that the renormalization $\mathcal R_h f$ is guaranteed to be defined on the cylinder that is only slightly smaller than $\Pi_{2h}$. Hence, without loss of generality we may assume that the neighborhood $\mathcal U_h$ is sufficiently small so that the following lemma holds (see \cite[Lemmas 3.1, 3.2]{GY}):

\begin{lemma}
  \label{lem-domain}
 For any $h>0$, the map  $\mathcal R_h$ is a real-symmetric complex-analytic operator $\mathcal R_h \colon \mathcal U_h \to \mathcal D_{1.5h}$.

\end{lemma}

\subsection{Hyperbolicity of renormalization}

\begin{definition}\label{Brjuno_def}
	
	An irrational number $\alpha\in(0,1)$ is a Brjuno number if the following sum converges:
	\begin{equation}
	\label{BY-function}
	\Phi(\alpha)=\sum_{n\geq 0}\alpha_{-1}\alpha_0\cdots\alpha_{n-1}\log\frac{1}{\alpha_n}.
	\end{equation}
	This sum is known as the Yoccoz-Brjuno function, see \cite{Yoc}. The set of all Brjuno numbers will be denoted by $\mathcal B$.
\end{definition}

For any $C>0$, consider the subset of Brjuno numbers $\mathcal B_C = \{\alpha\in \mathcal B \mid \Phi(\alpha)\le C\}$. For each $C>0$, this is a closed subset of $\bbR\setminus \bbQ$.

In the next theorem, we restrict $\mathcal R_h f$ to the cylinder $\Pi_h$ to get an operator $\mathcal R_h \colon \mathcal U_h \mapsto \mathcal D_h$.

\begin{theorem}[\cite{GY}, N.Goncharuk, M.Yampolsky]
\label{th-renorm}
For all sufficiently large $h>0$, the renormalization operator $\mathcal R_h \colon \mathcal U_{h}\to \mathcal D_{h}$ satisfies the following properties:
\begin{enumerate}
 \item $\mathcal R_h$ is a real-symmetric complex-analytic operator with compact differential at each rigid rotation $R_\alpha\in \mathcal T$;
 \item  \label{it-hyp} For each $C$, for $h>c_1 C+c_2$ where $c_1, c_2$ are universal constants, the renormalization $\mathcal R_h$ is uniformly hyperbolic on the set
 $$
 \{R_\alpha \mid \alpha\in \mathcal B_{C}\};
 $$
 moreover,
 \begin{itemize}
  \item  its unstable direction at each point of this set has complex dimension~$1$ and is tangent to the family $\{R_{a}, a\in \bbC/\bbZ\}$. The rate of expansion along the unstable direction is bounded from below by a universal constant.
 \item    The germ of a stable leaf at any $R_\alpha$  with $\alpha \in \mathcal B_{C}$ is a local codimension 1 analytic submanifold $\mathcal V_\alpha$. The rate of contraction along $\mathcal V_\alpha$ is bounded from above by a universal constant. 

 \item The submanifold $\mathcal V_\alpha$ contains only maps $f$ that are analytically conjugate to the rotation $R_\alpha$: there exists a conjugacy $\xi$, defined in $\Pi_{0.4 h}$, such that $\xi(0)=0$ and $f = \xi R_\alpha \xi^{-1}$. Moreover, if $f$ is analytically conjugate to $R_\alpha$ in the substrip $\Pi_{h/3}$ and sufficiently (depending on $\alpha$) close to $R_\alpha$, then  $f \in \cV_\alpha$.

 \end{itemize}
\end{enumerate}

\end{theorem}

Since Theorem \ref{th-main} is formulated for bounded-type rotation numbers, we will need the following observation.
\begin{lemma}
\label{lem-B-bound}
For any $k$, the Brjuno function $\Phi$ is bounded on irrational numbers $\alpha$ of type bounded by $k$. 
\end{lemma}
This follows from the estimate $\alpha_n>1/(k+1)$ for all $n$.

\section{Proof of the Theorem \ref{th-main}}
\subsection{Plan of the proof}

The idea of the proof is the following. We will show that when renormalization is applied to a map $f_t$ from the family $\mathcal F$, a certain M\"obius transformation is applied to its complex rotation number $\tau(F_t)$, where $F_t$ is the corresponding lift of $f_t$ to the real line. After $n$ iterates of renormalization, the germ of the family $\mathcal F = \{f_t\}$ at $t_0$ will turn into a family $\{g_t\}$ that is $\mu^n$-close to the family of rotations, where $\mu<1$ is the 
ratio between the contraction rate along the stable manifolds and the expansion rate along the unstable manifold of the renormalization operator.
% contraction rate along the stable manifolds of the renormalization operator. 
This follows from the hyperbolicity of the renormalization operator. But a pure rotation has no bubbles, thus the bubbles in the family $\{g_t\}$ are $\mu^n$-small. This enables us to estimate sizes of bubbles for the initial family $\mathcal F$ near $t=t_0$. 

In Sec. \ref{sec-renorm-bubbles}, we prove that  the renormalization operator acts as a M\"obius map on complex rotation numbers.  For this proof, we need a more explicit construction for complex rotation numbers of hyperbolic circle diffeomorphisms. This construction first appeared in \cite{NG2012-en}; it will be presented in section Sec. \ref{sec-expl}.

Sec. \ref{sec-proof} contains the proof of Theorem \ref{th-main}.

\subsection{Complex rotation numbers under renormalization}
\label{sec-renorm-bubbles}

The results of this subsection will be stated and proved for a more general version of the renormalization operator that will be denoted by $\mathcal R$. In this section we assume that $n=q_m$ for an arbitrary index $m$. Let $f$ be an analytic circle diffeomorphism and let $R\subset \bbC/\bbZ$ be any fundamental domain of $f$, such that $R$ is a curvilinear quadrilateral with two opposite sides being $I$ and $f^n(I)$. Here $I=[-ai,ai]$ is some vertical line segment. We define $\mathcal Rf:= \tilde\Psi\circ P\circ \tilde\Psi^{-1}$, where $P$ is the (partially defined) first return map of $f$ to $R$ and $\tilde\Psi\colon R/f^n\to \bbC/\bbZ$ is an arbitrary real-symmetric analytic diffeomorphism of the Riemann surface $R/f^n$ onto its image in $\bbC/\bbZ$. We will say that $\mathcal Rf$ is well-defined if the above construction yields an analytic circle diffeomorphism  $\mathcal Rf$.

The main result of this subsection is given by the following theorem:

\begin{theorem}
	\label{th-renorm-bubbles}
	Assume, $f$ is an analytic hyperbolic circle diffeomorphism with rotation number $\rot f = p/q\sim\{p_l/q_l\}\in [0,1)$ and a well defined renormalization $\mathcal Rf$. Let $F$ and $\mathcal R F$ be the lifts of $f$ and $\mathcal R f$ respectively, such that $\rot F = p/q$ and $\rot(\mathcal R F) \in [0,1)$. Assume that $n=n(f)=q_m$, for some index $m\ge 0$. Then 
	$$
	\tau(\mathcal R F) = \begin{cases}
		T_{p/q}\big(\tau(F)\big) & \text{if }m\in 2\bbZ+1\\
		T_{p/q}\big(\overline{\tau(F)}\big) & \text{if }m\in 2\bbZ,
	\end{cases}
	$$
	where 
	$$
	T_{p/q}(\tau)= \frac{-q_{m+1} \tau + p_{m+1}}{q_{m}\tau - p_{m}}.
	$$
	% that takes $p_{n-1}/q_{n-1}$ to $\infty$, $p_n/q_n$ to $0$, and $p_{n+1}/q_{n+1}$ to $\pm 1/a_n$.
\end{theorem}

Observe that the result of Theorem~\ref{th-renorm-bubbles} is independent of a particular choice of the renormalization operator and depends only on $m$. Indeed, all renormalizations $\mathcal Rf$ with the same $m$ will be analytically conjugate on the circle, and the complex rotation number is invariant under such conjugacies (Lemma \ref{analyt_conjugacy_lemma}).

\begin{remark}
Only even values of $m$ were used in Theorem \ref{th-renorm}, but for the sake of completeness, we provide the formula for odd $m$ in Theorem~\ref{th-renorm-bubbles} as well. 
\end{remark}

\begin{remark}
Note that when $\tau$ is real and coincides with the regular rotation number (i.e., $\tau=p/q$), then $T_{p/q}$ acts on $\tau$ as an iterate of the Gauss map: $T_{p/q}(\tau) = G^{m+1}(\tau)$. For complex values of $\tau$ this relation can be generalized as follows: for every positive integer $k\in\bbN$, let the function $G_k\colon\bbC\setminus\{0\}\to\bbC$, given by
$$
G_k(\tau):= \frac{1}{\tau}-k,
$$
be the analytic extension of the appropriate branch of the Gauss map. Then 
$$
T_{p/q}(\tau) =G_{k_m}\circ G_{k_{m-1}}\circ\dots\circ G_{k_0}(\tau)
$$
where $k_j$ are the terms of the continued fractional expansion of $p/q$. 
In particular, this implies that $T_{p/q}\in PSl(2, \bbZ)$.
\end{remark}

The proof of Theorem~\ref{th-renorm-bubbles} will rely on the explicit construction of the complex rotation number, described in the next subsection.

\subsubsection{Explicit construction for $\tau(F)$}
\label{sec-expl}

As above, let $f$ be an analytic circle diffeomorphism and let $F$ be its lift to the real line. Assume, $f$ is a hyperbolic difeomorphism. Then Theorem~\ref{bubble_theorem} states that the complex rotation number $\tau(F)$ has a strictly positive imaginary part. The explanation of this phenomenon is the following. It turns out that  the complex torus $T_{f+\omega}$ defined in the beginning of Section~\ref{complex_rot_number_intro_sec} does not degenerate as $\omega\to 0$: the hyperbolic circle diffeomorphism $f$ has a fundamental domain $A_{f}$ in a neighborhood of $\bbR/\bbZ$, and the complex rotation number $\tau(F)$, which is the limit of $\tau_F(\omega)$  as $\omega\to 0$, is the modulus of the complex torus $A_{f} / f$. 
The lower border of the fundamental domain $A_f$ will be a curve that passes below attracting periodic orbits and above repelling periodic orbits of $f$. Below we provide the construction of this fundamental domain, see also \cite{XB_NG}. 

Suppose that  $f$ is a hyperbolic diffeomorphism with $2n$ periodic orbits of period $q$ on the circle. The attracting and repelling periodic points alternate on the circle. Choose a linearizing chart of $f^q$ at each periodic point so that the charts around the points of the same orbit are obtained from each other as the push forward by the appropriate iterate of $f$. Let $a_j$ with $-\infty<j<+\infty$ be the lifts of these periodic points to the real line $\bbR$, enumerated consecutively from left to right, and let $\psi_j$ be the lifts of the corresponding linearizing charts. Note that each chart $\psi_j$ is defined on a complex neighborhood containing the open interval $(a_{j-1},a_{j+1})$. 

\begin{definition}
	We will say that a $1$-periodic piecewise smooth curve $\gamma\subset \bbC$ is \emph{suitable} for finding $\tau(F)$ if
	\begin{itemize}
		\item  the set  $\gamma\setminus \bbR$ is a union  of $2nq$ curves $\gamma_j$ such that $\gamma_j$ is an arc of a circle in the linearizing chart $\psi_j$ of the orbit of $a_j$. The circle is not necessarily centered at $a_j$ or on the real line; also, while the circular arc $\gamma_j$ must be contained in the domain of the linearizing chart, the circle itself is not necessarily completely contained there. 
		
		\item  $\gamma$ passes above repelling periodic points and below attracting periodic points of $f$ on the real line,
		
		\item  $F(\gamma)$ is above $\gamma$ in $\bbC$.
		
	\end{itemize}
\end{definition}
To construct a suitable curve, let $N$ be such that $F(a_j)=a_{j+N}$. Choose  endpoints $b_j\in (a_{j-1}, a_{j})$ and $b_{j+1}\in (a_j, a_{j+1})$ of $\gamma_j$ so that  $[F(b_{j}), F(b_{j+1})] \supset [b_{j+N}, b_{j+N+1}]$ if $a_j$ is repelling, and $[F(b_{j}), F(b_{j+1})] \subset [b_{j+N}, b_{j+N+1}]$ if $a_j$ is attracting. Finally, choose $\gamma_j$ with endpoints $b_j, b_{j+1}$ to be circle arcs in the linearizing charts, all having the same small angular size, above repellers and below attractors as required. This guarantees that  $f(\gamma_j)$ is above $\gamma_j$. Note that we can construct a suitable curve arbitrarily close to the real axis. 

\begin{figure}
\includegraphics[width=0.47\textwidth]{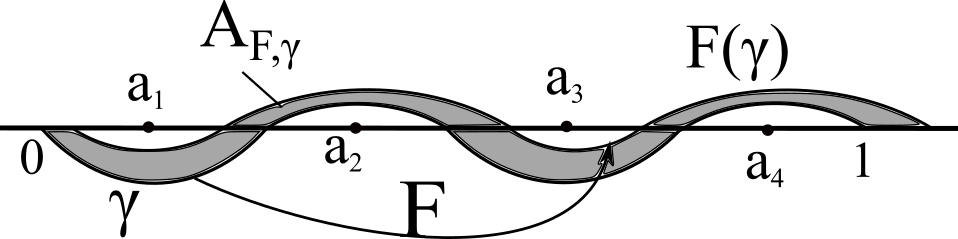}\hfil\includegraphics[width=0.47\textwidth]{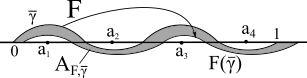}
\caption{Suitable (left) and anti-suitable (right) curves for the map $f$ with rotation number 0.5. The points $a_1, a_3$ and $a_2,a_4$ are lifts to $\bbR$ of the attracting and repelling periodic orbits of $f$ of period 2. }\label{fig-Buff}
\end{figure}

For a suitable curve  $\gamma$, let  $A_{F, \gamma}\subset \bbC$ be a fundamental domain of $F$ between  $\gamma$ and $F(\gamma)$.
Consider the complex torus $$E_\gamma(F)=A_{F, \gamma}/ (z\sim F(z), z\sim z+1).$$ It has two naturally  marked generators. Namely, we fix a point $x\in \gamma\cap \bbR$ and let the generators of this complex torus be the projections to $E_{\gamma}(F)$ of the curve  $\gamma$ and the curve that joins $x$ to $F(x)$ in $ A_{F, \gamma}$.

The following result is a refined version of the part of  Theorem \ref{bubble_theorem}  that concerns hyperbolic diffeomorphisms.
\begin{theorem}[\cite{NG2012-en}]
	\label{th-bubbles}
	Let $f$ be a hyperbolic analytic circle diffeomorphism, and let $F$ be its lift to the real line. Then the complex rotation number $\tau(F)$ equals the modulus of the complex torus $E_\gamma(F)$ for any suitable curve $\gamma$.
\end{theorem}

We will say that a curve $\gamma$ is \textit{anti-suitable} for $F$ if its complex conjugate is a suitable curve for $F$. The corresponding complex torus $E_\gamma(F)$ for an anti-suitable curve $\gamma$ is defined in exactly the same way as before. Note that $F(\gamma)$ is below $\gamma$ in this case.  If we consider generators of this torus with orientation, we can see that they are oriented in a non-standard way and thus the modulus of $E_{\gamma}(F)$ for an anti-suitable $\gamma$ is located in the lower half-plane. Note that the moduli of $E_{\gamma}(F)$ and $E_{\overline{\gamma}}(F)$ differ by complex conjugation. 

Thus we have the following statement: 

\begin{lemma}\label{lem-bubbles-anti}
Let $f$ and $F$ be the same as in Theorem~\ref{th-bubbles}. Then the complex rotation number $\tau(F)$ is complex conjugate to the modulus of the torus $E_\gamma(F)$ for any anti-suitable curve $\gamma$.
\end{lemma}

\subsubsection{Proof of Theorem~\ref{th-renorm-bubbles}}

\begin{proof}
In the course of the proof we will assume that the diffeomorphism $f$, and thus $\mathcal Rf$, is hyperbolic. Otherwise, according to Theorem~\ref{bubble_theorem}, the complex rotation numbers $\tau(F)$ and $\tau(\mathcal R F)$ coincide with the regular (real) rotation numbers $\rot F$ and $\rot \mathcal R F$ respectively, and the relation $\tau(\mathcal R F)= T_{p/q}(\tau(F))$ becomes obvious.

Consider a renormalizable hyperbolic map $f$ with $n(f) = q_m$, and the curve $\gamma\subset \bbC$ suitable for it (in particular, $\gamma$ is $1$-periodic). For each $k\ge 0$, let $A_k$ be the strip in $\bbC$ between $F^k(\gamma)$ and $F^{k+1}(\gamma)$, provided that this strip is defined.  Choose $\gamma$ close to the real axis so that $A_k$ is defined for all $k\le q_{m+1}+2q_{m}$, and $F$ is univalent in $A = \bigcup_{k=0}^{q_{m+1}+2q_{m}} A_k$ (we will possibly have to choose $\gamma$ even closer to $\bbR$ in what follows).

Consider the complex torus $T = A / (z\sim F(z), z\sim z+1)$. Since $A_k=F(A_{k-1})$ and $F$ is univalent in $A$, this torus coincides with $A_0/(z\sim F(z), z\sim z+1)$; since $\gamma$ is suitable for $F$, the modulus of the torus $T$ is equal to the complex rotation number $\tau(F)$. The uniformizing map $\Xi$ that takes $T$ to a standard torus, lifts and extends by iterates of $F$ to the map $\psi \colon A \to \bbC$ that conjugates $F$ and $z\mapsto z+1$ to shifts by $\tau(F)$ and by $1$ respectively.  

Let $F_m = F^{q_{m}}-p_{m}$ and $F_{m+1} = F^{q_{m+1}}-p_{m+1} $. 
Construct the curves $\xi_1, \xi_2\in A$ such that: 
\begin{itemize}
\item The curves $\xi_1, \xi_2, F_{m}(\xi_2), F_{m+1}(\xi_1)$ form a non-self-intersecting boundary of  a curvilinear rectangle $W$ in $A$;
\item  This rectangle belongs to $R\cup F_{m}(R)\cup F_m^{-1}(R)$, where $R$ is the domain of definition of the uniformizing map $\Psi$ used in the definition of renormalization;
\item $\Psi(\xi_1)$ is a suitable curve for computing $\tau(\mathcal Rf)$ for odd $m$, and an anti-suitable curve for even $m$.
\end{itemize}

We postpone the construction of $\xi_1, \xi_2$ and complete the proof of the theorem modulo this construction.  

Note that $\Xi$ conjugates $F_m$ and $F_{m+1}$ to the shifts by $ q_m\tau(F)-p_m$ and $ q_{m+1}\tau(F)-p_{m+1}$;  due to our assumptions on $\xi_1, \xi_2$, the rectangle $\Xi(W)$ is a fundamental domain for these shifts, thus $W$ is a fundamental domain for $F_m, F_{m+1}$ in $A$. Also, the modulus of the torus
$$
E = A/(F_m, F_{m+1}) = W/(F_m, F_{m+1}) 
$$
with generators $\xi_1, \xi_2$ equals $\frac{q_{m+1}\tau(F) - p_{m+1}}{q_{m} \tau(F)-p_{m}}=-T_{p/q}(\tau(F))$.

Recall that $\Psi $ conjugates $F_m$ to $z\to z-1$ and is defined in $R$; here $F$ is close to  rotations and thus univalent in $\Pi_h$. Hence $\Psi$ extends to $R\cup F^{m}(R)\cup F^{-m}(R)$ via the relation $\Psi(F^m(z))=\Psi(z)-1$. So $\Psi$ is defined on a neighborhood of $W$.   Since $F_{m+1}$ coincides with the first-return map to $[0, F_m(0)]$ on a neighborhood of $F_m(0)$, the map  $\Psi$ conjugates $F_{m+1}$ to the lift of $\mathcal Rf$ that satisfies $ \mathcal RF(-1) \in [-1, 0)$, i.e. to the lift from the statement of the theorem. Since the map $\Psi$ is defined in $W$, it  descends to a biholomorphism between the torus $E$ and $$\tilde E =  \Psi(W) / (z\to z-1, z \to \mathcal RF(z)).$$
Thus the modulus of the torus $\tilde E$ also equals $-T_{p/q}(\tau(F))$.

Finally, $\Psi(\xi_1)$ is a suitable curve for odd $m$  for computing $\tau(\mathcal R F)$  (and anti-suitable for even $m$). The second generator $\Psi(\xi_2)$ joins a certain point $y\in \Psi(\xi_1)$ to $\Psi F_{m+1}\Psi^{-1}(y) = \mathcal RF(y)$. Thus the complex torus $\tilde E$ coincides with $E_{\Psi(\xi_1)}(\mathcal RF)$, modulo the sign change of the first generator. We conclude that the modulus of the torus $E_{\Psi(\xi_1)}(\mathcal RF)$ equals $T_{p/q}(\tau(f))$. 

Due to Theorem \ref{th-bubbles} and Lemma \ref{lem-bubbles-anti}, this completes the proof, modulo construction of $\xi_1, \xi_2$.

\begin{figure}[h]
	\includegraphics[width=\textwidth]{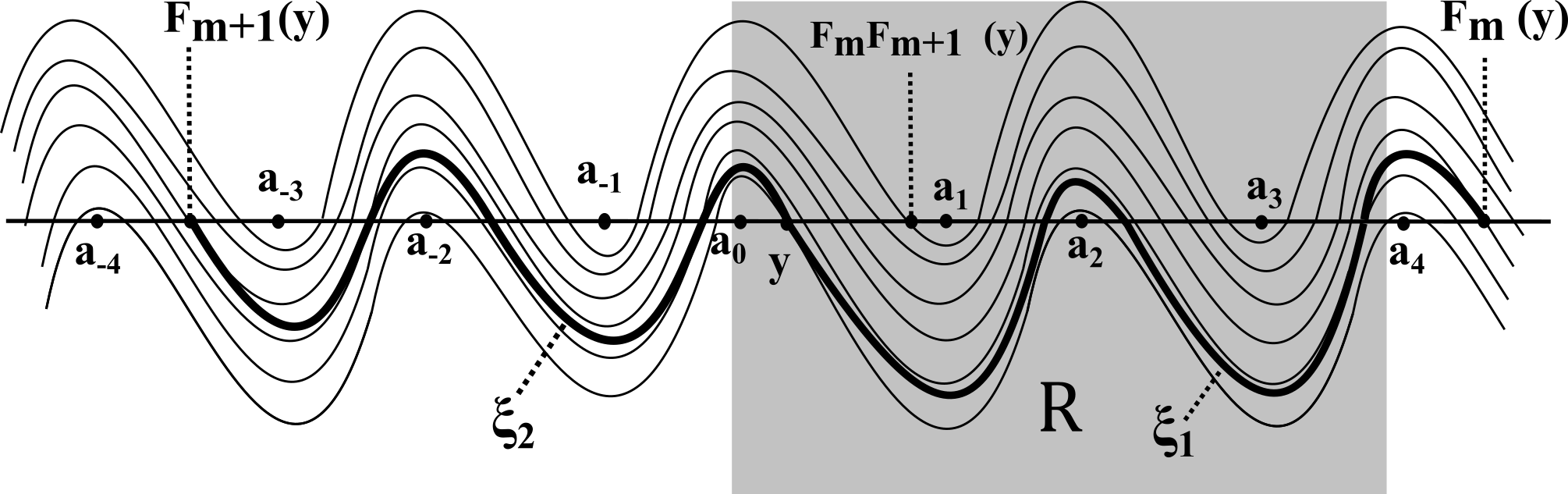}
	\caption{The strip $A$, the domain $R$ (shadowed), and the curves $\xi_1, \xi_2$ (both shown in thick). Here $\rot f = 3/5$, $p_m/q_m = 1/2$, and $p_{m+1}/q_{m+1}=2/3$.}\label{Strip_A_Fig}
\end{figure}

\textbf{Construction of $\xi_1, \xi_2$.}

Assume $m$ is even, so $F^{q_m}(0)-p_m>0$. 
Let $a_j\in \bbR$ be the lifts of periodic points of $f$ to $\bbR$, ordered from left to right. 
%Note that $A\setminus \bbR$ is a union of curvilinear rectangles that belong to the domains of the linearizing charts $\psi_j$ of $a_j$.

%Let $R$ be the fundamental domain of renormalization, as defined in the beginning of subsection~\ref{renorm_subsection}, and let $U\supset R$ be its open neighborhood.
Since $q_m<q$, the interval $[0, F^{q_m}(0)-p_m]$ contains periodic points of $f$; let $a_0$ be the leftmost of these points. Take a point $y\in\gamma\cap \bbR$ that is closest to $a_0$ from the right; we have $y\in R$. Let $y_{1}:=F_m(y)\in F_m(R)$, $y_{1}\in A_{q_m-1}\cap A_{q_m}$. Note that $y_{1}$ belongs to the linearizing chart of a periodic point $a_s = F^{q_{m}}(a_0)-p_{m}=F_m(a_0)$. Join $y$ to $y_{1}$  by a curve $\xi_1\subset A$ that intersects each linearizing domain of $a_j$ for $j=0, 1, \dots, s$ by an arc of a circle and visits $A_0, A_1,\dots, A_{q_{m}-1}$ subsequently. If $\gamma $ was chosen sufficiently close to the real line, then $\xi_1$ stays in $R\cup F_m(R)$.

Let $\xi_2$ be constructed in a similar way to join $y$ to $y_2 = F_{m+1}(y)$; note that $\xi_2$ belongs to $(F_{m})^{-1}(R)$ if $\gamma $ is sufficiently close to the real line. It is easy to see that $\xi_1, \xi_2, F_{m+1}(\xi_1),$ and $F_m(\xi_2)$ bound a fundamental domain $W$ that belongs to a neighborhood of a subinterval of   $[(F_m)^{-1}(0), F_m(0) ]$ and thus to $R\cup F_m(R)\cup (F_m)^{-1}(R)$.

It remains to prove that $\Psi(\xi_1)$ is a suitable curve for computing $\tau(\mathcal RF)$. The curve $\Psi(\xi_1)$ is a loop in $\bbC/\bbZ$. Since linearizing charts of the periodic orbits of $\mathcal R f$ differ from those for $f$ by conjugacies with $\Psi$, the curve $\Psi(\xi_1)$ consists of arcs of circles in the linearizing charts of the periodic points of $\mathcal Rf$.   On a neighborhood of $[0, F_m(0)]$, the first-return map $P$ to $R$ is given either by  $F_{m+1}$ or by $F_m\circ F_{m+1}$. Recall that $\xi_1$ visits only the strips $A_k$ with $0<k<q_{m}-1$; thus $P(\xi_1)$ is located in the union of the strips $A_k$ with $q_{m+1}<k<q_{m+1}+2q_{m}$, hence is above $\xi_1$ in $A/ F_m$. 
Since $m$ is even, $\Psi$ reverses orientation, thus $\Psi(\xi_1)$ is anti-suitable for computing  $\tau(\mathcal Rf)$.

The proof for the case of an odd $m$ is similar. In that case, $F_m(0)$ is to the left from $0$; also, $\xi_1$ will go to the left, joining $y$ to $F_m(y)$, and $\xi_2$ will go to the right. Since $\Psi$ preserves orientation, it takes $\xi_1$ to a suitable curve in this case.

This completes the proof.

\end{proof}

Recall that $D_{\eps, r}$ is the  disc of diameter $\eps$ in the upper halfplane that is tangent to $\bbR$ at $r$. We will need the following general lemma that will be applied in the more specific case of the maps $T=T^{-1}_{p/q} \in PSl(2,\bbZ)$.

\begin{lemma}
\label{lem-PSL}
Assume, $k,l\in\bbZ$ are relatively prime integers and $T\in PSl(2, \bbZ)$ is a M\"obius transformation, such that $T(k/l) = \tilde k/\tilde l \neq\infty$, where the integers $\tilde k$ and $\tilde l$ are relatively prime. Then for any $\eps>0$, the map $T$ takes the disk $D_{\eps |l|^{-2}, k/l}$ to the disk $D_{\eps |\tilde l|^{-2}, \tilde k/\tilde l}$.
\end{lemma}
\begin{proof}
Consider the map  $T= \frac{az+b}{cz+d}$, where $ad-bc=1$. Put  $\tilde k = ak +bl$, $\tilde l = ck+dl$; then $\tilde k/\tilde l = T(k/l) $.
Since
$k = d\tilde k - b\tilde l$ and $l = -c\tilde k + a\tilde l$, 
it follows that the greatest common divisor of $k$ and $l$ coincides with the one for $\tilde k$ and $\tilde l$. Thus, $\tilde k$ and $\tilde l$ are relatively prime.

 For any $A,B,C,D\in \bbR$, $C\neq 0$, the image of the half-plane $\{\Im z > 1/\eps\}$ under the map $g(z)=\frac{Az+B}{Cz+D}$ is a disk of diameter $\eps |BC-AD|\cdot |C|^{-2}$ that is tangent to $\bbR$ at $A/C$. Indeed,
 $$
 g(z)=\frac{Az+B}{Cz+D} = \frac{A}{C} + \frac{(BC-AD)/C^2}{z + D/C}
 $$ 
 so $g$ is a composition of the shift by $D/C$, followed by the map $1/z$, multiplication by $(BC-AD)/C^2$,  and the shift by $A/C$, which implies the statement.

 The disc $D_{\eps |l|^{-2}, k/l}$ itself is the image of the half-plane $\{\Im z > 1/(\eps |l|^{-2})\}$ under the map $z\mapsto -1/z+k/l$.
We get that $T(D_{\eps |l|^{-2}, k/l})$ coincides with the  image of $\{\Im z > 1/(\eps |l|^{-2})\}$ under the composition 
$$
z\mapsto \frac{a(-1/z+k/l)+b}{c(-1/z+k/l) + d} = \frac{-a+z (a (k/l)+b)}{-c + z(c (k/l) +d) }.
$$
According to the argument from the previous paragraph, $T(D_{\eps |l|^{-2}, k/l})$ is the disk of diameter $\eps |l|^{-2}  |c(k/l)+d|^{-2} = \eps |\tilde l|^{-2}$, tangent to the real axis at the point $\frac{a (k/l)+b}{c (k/l) +d} = \frac{\tilde k }{\tilde l }$.
\end{proof}
\subsection{Proof of Theorem \ref{th-main}}
\label{sec-proof}

\begin{proof}[Proof of Theorem \ref{th-main}]

Switching to the linearizing chart of $f_{t_0}$, we will now work with a family $f_t$ that passes through the rotation $f_{t_0}=R_\alpha$. Note that this family is strictly monotonic: $\frac{\partial }{\partial t} f_t>0$. Suppose that $f_t$ for  $t\approx t_0$ are defined in a strip of width $h_1$ around $\bbR/\bbZ$.

Without loss of generality, we may assume that $t_0=0$ in the statement  of Theorem~\ref{th-main}.

\textbf{Step 1: Reduction to the case of a wide strip.}

 First, let us pass to the case $h_1>h$, where $h$ satisfies the assumptions of Theorem \ref{th-renorm}, i.e. $h>c_1\Phi(\beta)+c_2$ for all $\beta$ of type bounded by $k$; then we will be able to use the result of Theorem \ref{th-renorm} on hyperbolicity of the renormalization operator.

To this end, recall that our renormalization operator $\mathcal R_{h_1}\colon \mathcal U_{h_1}\to \mathcal D_{h_1}$ was defined as the restriction of the map  $\Psi P \Psi^{-1}$ to the strip of width $h_1$. However, with the proper choice of $\Psi$, we can guarantee that $\Psi P \Psi^{-1} $ is defined in a much wider strip. Let $l_s=\alpha_0\alpha_1\dots\alpha_{s-1}$ be the length of the  fundamental domain $[0, (R_{\alpha})^{q_s}(0)-p_s]$ for the rotation $R_\alpha$. For any $f\in \mathcal D_{h_1}$ close to $R_\alpha$, consider a curvilinear rectangle bounded by $I^h = [-0.99ih, 0.99ih]$, $f^{q_s}(I^h)$, and two straight line segments joining their endpoints. Let $\Psi$ be any biholomorphic map in $R$ that conjugates $f^{q_s} $ to the shift by $(-1)$ and depends analytically on $f$.  For $f$ close to the rotation $R_\alpha$, the map $\Psi P \Psi^{-1}$ is defined on the strip of width almost $0.99 h_1/l_s$.  Find an integer $s$ so that we have $c_1\Phi(\beta)+c_2 <0.5 h_1 / l_s$, where $c_1, c_2$ are the same as in Theorem \ref{th-renorm}, for all $\beta$ of type bounded by $k$. This is possible since $l_s\to 0$ and $\Phi(\beta)$ is bounded (Lemma \ref{lem-B-bound}). 
Note that we can choose $s=s(k)$ uniformly for all $\alpha\in M_k$.

 Now, let $h=0.9 h_1/l_s$. Consider the corresponding renormalization operator $\mathcal R_{h_1, h, s} \colon \mathcal D_{h_1} \to \mathcal D_{h}$,  defined by $\mathcal R_{h_1, h, s} f = \Psi^{-1} P \Psi|_{\Pi_h}$ where $P$ is the first-return map to $[0, f^{q_s}(0)-p_s]$ as before.  The smooth family of circle maps  $g_t = \mathcal R_{h_1, h, s} f_t$, $g_0=R_{\alpha_s}$,   is defined (for small $t$) in a strip $\Pi_h$  with the property that $h>c_1\Phi(\beta)+c_2$ for all  $\beta$ of type bounded by $k$. 

The monotonicity of the family $g_t$ is not necessarily preserved, since the rescaling that appears in the renormalization can depend on $t$ in a nontrivial way. Because of that, we need to consider a weaker property that remains invariant under renormalizations. Note that any strictly  monotonic family satisfies $\int_{\bbR/\bbZ} \frac{df_t}{dt} dz > 0$, i.e. is transversal to the subspace $\{v\in T\mathcal D_{h_1} \mid \int_{\bbR/\bbZ} v =0\}$ of the tangent space $T\mathcal D_{h_1}$.
\begin{lemma}
Let $\{f_t\}_{t\in I}\subset \mathcal D_{h_1}$ be a smooth family of analytic circle diffeomorphisms with $f_0=R_\alpha$, where $\alpha$ is an irrational number. If this family is transversal to the subspace $\{v\in T\mathcal D_{h_1} \mid \int_{\bbR/\bbZ} v =0\}$ at $t=0$, then the family $g_t=\mathcal R_{h_1, h, s} f_t$ is transversal to the subspace $\{v\in T\mathcal D_{h} \mid \int_{\bbR/\bbZ} v =0\}$ at $t=0$. 
\end{lemma}

\begin{proof}

Let $\Psi_t, P_t$ be the chart on the fundamental domain and the first-return map associated with $f_t$.
Using the fact that $P_0$ is a shift (hence, $P_0'=1$) and $\Psi_0$ is an affine map,  we get 
\begin{multline*}
\frac d {dt}\mathcal R_{h_1,h, s} f_t|_{t=0}= \frac{d}{dt}\Psi_t P_t \Psi_t^{-1} |_{t=0}=\\= \frac{d}{dt}\Psi_t|_{P_0\Psi_0^{-1}, t=0} - \frac{d}{dt}\Psi_t|_{\Psi^{-1}_0, t=0} + \Psi'_0 \frac{d}{dt}P_t|_{\Psi_0^{-1},t=0}.
\end{multline*}

 If we set $\xi = \frac{d}{dt}\Psi_t|_{\Psi^{-1}_0(z), t=0}$, the first two summands take the form $\xi(\Psi_0 P_0 \Psi_0^{-1}(z)) - \xi(z)$, and since $\Psi_0P_0\Psi_0^{-1} = \mathcal R_{h_1, h, s} R_{\alpha}$ is a shift, the integral of the sum of the first two summands over the circle is zero. Since $P_t$ is the first-return map to the fundamental domain of $f$, its derivative with respect to $t$  is strictly positive and equals the sum of the derivatives  $\frac{d}{dt}f_t$ along the orbit of the rotation $f_0$. The map $\Psi_0$ is a linear expansion that preserves orientation for odd $s$ and  reverses orientation for even $s$.  In the first case,  the last summand has a strictly positive integral over the circle, and in the second case, it has a strictly negative integral. In both cases, $g_t$ is transversal to $\{v\in T\mathcal D_{h} \mid \int_{\bbR/\bbZ} v =0\}$ at $t=0$.
 \end{proof}
 
Due to  Theorem \ref{th-renorm-bubbles}, a certain M\"obius map $T_{p/q}\in PSl(2,\bbZ)$ takes the $p/q$-bubble of the family $f_t$ to the corresponding bubble of the family  $g_t=\mathcal R_{h_1, h, s} f_t$ or to its complex conjugate. Note that in a small neighborhood  of $\alpha\notin \bbR/\bbZ$, $n(f)=q_m$ is constant and the first $m+1$ terms of the continued fractional expansion are also constant, thus the map $T_{p/q}$ is the same M\"obius map for all $p/q$ in this neighborhood.

Due to  Lemma \ref{lem-PSL}, it is now sufficient to prove Theorem \ref{th-main} for  analytic families of circle diffeomorphisms  $f_t, f_0 = R_\alpha,$ that are transversal to $\{v\in T\mathcal D_{h} \mid \int_{\bbR/\bbZ} v =0\}$ and defined in a strip $\Pi_{h}$ such that $h>c_1\Phi(\beta)+c_2$ for all $\beta$ of type bounded by $k$. 

In particular, assumptions of Theorem \ref{th-renorm} are satisfied for $h$ and any rotation number from the forward orbit of $\alpha$ under the Gauss map $G$. 

\textbf{Step 2: Expansion and contraction rates}
Let $\|\cdot\|_h$ be the sup-norm in the strip $\Pi_h$.
Put
$$M_k = \{R_\alpha\mid \alpha\text{ of type bounded by }k\}.$$

Note that with our choice of the iterate $n=n(f)=q_m$ in the definition of $\mathcal R_h$, the derivative $(G^{m+1})'$ of the corresponding iterate of the Gauss map is uniformly bounded below on $[0,1]$, $|(G^{m+1})'|>\lambda>1$. Since $\mathcal R_h$ acts on rotations by the $m+1$-power of the Gauss map,  $$|d|_{R_\beta} \mathcal R_h \, 1| > \lambda$$ for all $\beta\notin K$.

Due to \cite{GY}, the stable distribution of $\mathcal R_h$ on the set $\{R_\alpha\mid \alpha\in\mathcal B\}$ is the codimension-1 subspace $\{v\in T\mathcal D_{h} \mid \int_{\bbR/\bbZ} v =0\}$. Furthermore, according to \cite[Theorem 4.6]{GY}, for any $h$ larger than a universal constant, for any vector field $v\in T\mathcal D_h$ satisfying $\int_{\bbR/\bbZ} v dz=0$, and for any $\alpha\in \mathcal B$, there exist a positive integer $l=l(\alpha, h)$ and a positive constant $c=c(\alpha,h)$, such that
$$
    \|d|_{R_\alpha}\mathcal R_h^l v\|_h\le c(\alpha, h)\cdot  0.1^l \|v\|_h.
$$ 
The corresponding universal constant is included in our restriction $h>c_1\Phi(\beta)+c_2$, hence the condition on $h$ is satisfied.   
Thus for each point $R_\alpha\in M_k$, we can choose $l$ so that  $\int v dx=0$ implies $\|d|_{R_\alpha}\mathcal R_h^l v\|_h< 0.5 \|v\|_h$. Since $\mathcal R_h$ is analytic on a neighborhood of a compact set $M_k$, we can choose $l$ uniformly for all $\beta\in M_k$. Thus, assuming that a sufficiently large value of $h$ is fixed, we conclude that there exist $l=l(k)$ and $\tau = \tau(k) = \sqrt[l]{0.5}<1$, such that for each  $R_\alpha\in M_k$ and for any $v$ with  $\int_{\bbR/\bbZ} v dx=0$, we have 
\begin{equation*}
%\label{eq-contract}
\|d|_{R_\alpha}\mathcal R_h^l v\|_h< \tau^l \|v\|_h \text{ and } d|_{R_\alpha}\mathcal R_h^l v \text{ has zero average on }\bbR/\bbZ. 
\end{equation*}

\vskip 1 cm

The next two steps of the proof are known as ``the inclination lemma'' in the hyperbolic theory, see \cite{Palis}. The inclination lemma states that under the iterates of a hyperbolic  map, the surfaces that are transversal to the stable foliation will tend to the leafs of the unstable foliation in $C^1$ metric exponentially quickly. We did not find a suitable reference for the Banach space setting; the arguments below are heavily based on \cite{Walter} where a similar statement was proved for inverse iterates of operators in Banach spaces that have periodic hyperbolic points.

For the sake of simplicity, starting from this moment, we will work with the real-symmetric space of circle diffeomorphisms $\mathcal D_h^{\bbR}$ and the corresponding tangent bundle $T\mathcal D_h^{\bbR}$ as opposed to their complex versions.

In the rest of the proof, we will omit $h$ and write $\mathcal R$, $\|\cdot\|$ instead of $\mathcal R_h$, $\|\cdot \|_h$ respectively.

\textbf{Step 3: Slopes of curves: one step of  renormalization}

Let $p_u\colon T\mathcal D_h^{\bbR} \to \bbR$ be the operator $p_u\colon v\mapsto \int_{\bbR/\bbZ} v dz$, and let $p_s\colon T\mathcal D_h^{\bbR} \to \{v \in T\mathcal D_h^{\bbR}, \int_{\bbR/\bbZ} v=0\}$ be the operator $p_s \colon v \mapsto v-p_u v$. Here $s,u$ stand for ``stable'' and ``unstable''.  Let $$s(v) = \frac{\|p_s(v)\| }{ |p_u(v)|}$$ be the \emph{slope} of the vector. 

Let $P_u\colon \mathcal D_h^{\bbR} \to \bbR$ be given by $P_u f = \int_{\bbR/\bbZ} (f(z)-z) dz.$ Let   $P_s f = f- P_u f -\id$; then $P_s f$ is a 1-periodic map on $\Pi_h$ with zero integral over $\bbR/\bbZ$. Note that $\|P_s f\|$ is a distance between $f$ and rotations. 

Recall that $\tau^l<1$ is a uniform estimate on the contraction  rate of $\mathcal R^l$ over $M_k$, and $\lambda>1$ is a uniform estimate on the expansion rate of $\mathcal R$ near rotations. Let
$$
\mu = \frac{\tau^l}{\lambda^l}.
$$

\begin{lemma}
For any $\delta>0, \eps>0$ there exists a sufficiently small neighborhood $\mathcal U$ of $M_k$ in $\mathcal D_h^{\bbR}$, such that for any $f$ in this neighborhood and any  $v\in T\mathcal  D_h^{\bbR}$ with $s(v)<1/\delta$,  we have 
\begin{equation}\label{eq-1step}
s(d_f\mathcal R^l v) \le (\mu+\eps) s(v)+C \|P_s f\|,
\end{equation}
where $C>0$ is a constant that is independent from $\eps$ and $\delta$.

\end{lemma}
\begin{proof}

Let $a=\int_{\bbR/\bbZ} v dx$, then  $v= a\cdot 1 + w$ with  $\int_{\bbR/\bbZ} w dz=0$; note that $$s(v)=\frac{\|w\|}{|a|}.$$ Let $b=P_u f\in \bbC$.

  Assume that $\mathcal U$ is a sufficiently small neighborhood of the compact set $M_k$, so that $\mathcal R^l$ is analytic on $\mathcal U$. Then there is a uniform constant $C_1>0$, such that for any $f\in\mathcal U$, we have
%  Due to the analyticity of $\mathcal R$ on a neighborhood of a compact set $M_k$, we have a uniform constant $C$ such that
  $$\|d_f\mathcal R^l 1  - d_{R_b}\mathcal R^l 1 \| \le C_1\|f-R_b\| = C_1\|P_s f\|.$$
  Furthermore, for an arbitrarily small $\eps_1>0$, one can still assume that the neighborhood $\mathcal U$ is sufficiently small, so that
  $$\|d_f\mathcal R^l w  - d_{R_\beta}\mathcal R^l w \| \le C_1 \|f-R_\beta\|\cdot  \|w\| \le \eps_1 \|w\|$$ for some  $\beta\in M_k$. %; the last inequality holds if $\mathcal U$ is sufficiently small.
  
   We conclude that $$d_f\mathcal R^l v  = d_f\mathcal R^l (a\cdot 1+w)  = a\cdot d_{R_b}\mathcal R^l 1 +au_1 + d_{R_\beta}\mathcal R^l w +u_2,$$ where $\|u_1\|\le  C_1 \|P_s f\|  $ and $\|u_2\|\le \eps_1\|w\|$. Due to the estimates on the expansion and contraction rates, we obtain: 
 
  \begin{itemize}
  \item $p_u d_f\mathcal R^l v =a L \cdot 1 + a p_u u_1 + p_u u_2$ 
  \end{itemize}
  
  with $|L|\ge \lambda^l$, $|a p_u u_1|\le  C_1|a| \cdot \|P_s f\|$, and $\|p_u u_2\|\le \eps_1 \|w\|$; 
 \begin{itemize} 
\item   $p_s d_f\mathcal R^l v =a p_s u_1 + p_s  d_{R_\beta}\mathcal R^l w  +p_s u_2$ 
 \end{itemize}
with $|a p_s u_1|\le   C_1 |a|\|P_s f\|$ and $\|p_s  d_{R_\beta}\mathcal R^l w  +p_s u_2\|\le (\tau^l+\eps_1) \|w\|.$
 
  Thus 
  \begin{multline*}s(d_f\mathcal R^l v) \le \frac{C_1|a| \|P_s f\| + (\tau^l+\eps_1)\|w\|}{|a|\lambda^l  - C_1|a| \|P_s f\| - \eps_1\|w\| } =\\= \frac{C_1 \|P_s f\| + (\tau^l+\eps_1) s(v)}{\lambda^l-C_1\|P_s f\| -\eps_1 s(v) }.
  \end{multline*}
  
Recall that $s(v)<1/\delta$, $\lambda>1>\tau$. By choosing $\eps_1$ and the neighborhood $\mathcal U$ sufficiently small,  we can guarantee that 
$$s(d_f\mathcal R^l v) \le\left( \frac{\tau^l}{\lambda^l}+\eps\right) s(v) + C\|P_s f\|,$$ 
where $C=2C_1/\lambda^l$.   
\end{proof}

\textbf{Step 4: Slopes of curves: iterating the renormalization operator.}

In the next lemmas, we will use a neighborhood $\mathcal U$ of the set $M_k$ that has a special form. Namely, we represent an open neighborhood of real rotations in $\mathcal D_h^{\bbR}$ as a product  $[0,1]\times U$ where $U$ is an open ball centered at zero in the space $\{f\in \mathcal D_h^{\bbR} \mid \int_{\bbR/\bbZ}(f(z)-z) dz= 0\}$. In other words, we represent $f\in \mathcal D_h^{\bbR}$ as $f=R_b+g$ where $b= \int_{\bbR/\bbZ} (f(z)-z) dz$.

A neighborhood $\mathcal U$ of $M_k$ in $\mathcal D_h^{\bbR}$ will have the form $\mathcal U = J \times U_\nu(0)$ where $J$ is a union of open intervals and $U_\nu(0)$ is a open ball of radius $\nu$ in the space $\{f\in \mathcal D_h^{\bbR} \mid \int_{\bbR/\bbZ}(f(z)-z) dz= 0\}$. In what follows, depending on the context, we will identify $J$ either with the union of intervals, or with the family of rotations by the angles from these intervals.

The next lemma is the suitable version of the inclination lemma. 

\begin{lemma}
\label{lem-inclination}
 For any $\delta>0, \eps>0$,   there exist $C_1, C_2>0$  such that for a sufficiently small neighborhood  $\mathcal U= J \times U_\nu(0)$ of $M_k$, the following holds. 

Suppose that a one-parameter smooth family of analytic circle maps $\{f_t, t\in I\} \subset \mathcal D_h^{\bbR}$ with $f_0 = R_\beta\in M_k$ satisfies $s(\frac{d}{dt}f_t)<1/\delta$ for all $t\in I$.

Suppose that $(\mathcal R^l)^j f_t \in \mathcal U$ for all $j=0, 1, \dots, r-1$ and all $t\in I$, where $r\ge 0$ is an arbitrary integer.

Then for $g_t = (\mathcal R^l)^{r} f_t$, we have  
$$s\left(\frac{d}{dt}g_t\right) < C_1 (\mu+\eps)^{r}$$ and $$\|P_s g_t\|< C_2 (\mu+\eps)^{r}.$$ 

\end{lemma}
\begin{proof}
Choose $\mathcal U$ in the form $\mathcal U=J\times U_{\nu}(0)$ using the previous lemma for $\eps/2$ instead of $\eps$. We will diminish $\mathcal U$ later in the proof.

We set $C_1 = 1+1/\delta$, $C_2 = \eps/(2C)$ where $C$ is the same as in the previous lemma. The proof goes by induction. The base $r=0$ is obvious if $\mathcal U$ is small enough.  Suppose that the statement holds for iterates $0, 1, 2, \dots, r-1$ of $\mathcal R^l$. 

Let  $v_t=\frac{d f_t}{dt}$. In the next computation, we will fix $t\in I$ and write $f, v$ instead of $f_t, v_t$. 
Let $f_j = (\mathcal R^l)^jf$. 

Then due to the previous lemma, 
\begin{multline*}
s\left(\frac{d}{dt}g_t\right)= s((d|_f\mathcal R^l)^{r} v)\le \left( \mu+\frac {\eps}2\right) s((d|_f\mathcal R^l)^{r-1} v) + C\|P_s f_{k-1}\| \le \dots \\ \le   \left(\mu+\frac{\eps}2\right)^{r} s(v) +C \left(\mu+\frac{\eps}2\right)^{r-1} \|P_s f_{1}\| + C\left(\mu+\frac{\eps}2\right)^{r-2}\|P_s f_{2}\| + \dots + C\|P_s f_{r-1}\|.  
\end{multline*}
Since $\|P_s f_{j}\|<C_2(\mu+\eps)^{j} $ due to the inductive statement, and $s(v)<1/\delta$, we get 
\begin{multline}
s\left(\frac{d}{dt}g_t\right)\le \frac{(\mu+\frac {\eps}2)^{r}}{\delta} + C C_2(\mu+\eps)^{r-1}  \left(1+\frac{\mu+\frac{\eps}2}{\mu+\eps} + \frac{(\mu+\frac{\eps}2)^2}{(\mu+\eps)^{2}}+ \dots +\frac{(\mu+ \frac{\eps}2)^{r-1}}{(\mu+\eps)^{r-1}}\right) < \\<\frac{(\mu+\eps)^{r}}{\delta}   + C C_2 \frac{(\mu+\eps)^{r-1}}{1 - \frac{\mu+\frac{\eps}{2}}{\mu+\eps}} < \frac{(\mu+\eps)^{r}}{\delta}   + C C_2 \frac{(\mu+\eps)^r}{\eps/2} =\\=(\mu+\eps)^{r} \left(\frac{1}{\delta} + 2C C_2/\eps\right) = (\mu+\eps)^{r}  \left(\frac 1 {\delta} + 1\right) =C_1 (\mu+\eps)^{r}. 
\end{multline}
%since $\tau^m+\eps >\mu$ in the previous lemma. 

Now, prove that $\|P_s g_t\|<C_2(\mu+\eps)^{r}$.    
We will integrate  the already established inequality $s(\frac{d}{dt}g_t) < C_1 (\mu+\eps)^{r} $ with respect to $t$. Since the slope is bounded, $p_u (\frac{d}{dt}g_t) = \int \frac{d}{dt}g_t dz \neq 0$; assume without loss of generality that this integral is positive for all $t\in I$.

The bound on the slope implies that for all $t$, $\|p_s \frac{d}{dt}g_t\|\le C_1(\mu+\eps)^r\|p_u  \frac{d}{dt}g_t\|$ i.e.  $$\sup_{\Pi_h }\left|\frac{d}{dt}g_t -\int_{\bbR/\bbZ} \frac{d}{dt}g_t dx\right| 
\le   C_1 (\mu+\eps)^{r}  \int_{\bbR/\bbZ} \frac{d}{dt}g_t dx.$$ Integrating over $t$, we get that for any $t>0$, $t\in I$, 
\begin{equation}
\label{eq-inclination} \sup_{\Pi_h} \left|g_t(x) - x-\int_{\bbR/\bbZ} (g_t(x)-x) dx\right| \le C_1(\mu+\eps)^{r} \left(\int_{\bbR/\bbZ} (g_t(x)-x) dx - \int_{\bbR/\bbZ} (g_0(x)-x) dx\right). 
\end{equation}
In the left side of the above inequality, we used the fact that for $t=0$, $g_0$ is a rotation. Now we get
\begin{equation}
\label{eq-inclination1} \|P_s g_t\|  \le C_1(\mu+\eps)^{r} \left(P_u(g_t)-P_u(g_0)\right). 
\end{equation}
Similar estimate holds for $t<0$.

Since $g_t$ belongs to $\mathcal U = J\times U_\nu(0)$ for all $t\in I$, the expression $\left(P_u(g_t)-P_u(g_0)\right)$  is bounded above by the length $\theta$ of the longest subinterval of $J$. Since $M_k$ is a nowhere dense set, by diminishing its neighborhood $\mathcal U$, we can achieve arbitrarily small $\theta$ and thus guarantee that $\theta C_1 < C_2$. 
\end{proof}

\textbf{Step 5: Size of the arc}

Recall that the action of  $\mathcal R$ on rotations is induced by the power of a Gauss map that depends on the rotation: if $n=n(f)=q_{m(f)}$, then $\mathcal R (R_\alpha)=R_\beta$, where $\beta =G^{m(f)+1} (\alpha) = \alpha_{m(f)+1}$, and  $m(f)$ is bounded on any closed set that does not contain rotations with rotation numbers $1/j, j\le 100$. Choose $m$ so that $m(f)\le m$ on a certain neighborhood of $M_k$; we will assume that our neighborhood $\mathcal U$ is small enough so that $m(f)\le m$ on $\mathcal U$. 
  
\begin{lemma}
\label{lem-arc}
Let $\{f_t, t\in I\}\subset \mathcal D_h^{\bbR}$ be a smooth family of analytic circle diffeomorphisms such that $f_0=R_\alpha\in M_k$ and the family $\{f_t\}$ is transversal to the subspace $\{v\in T\mathcal D_h^{\bbR} \mid \int_{\bbR/\bbZ} v dz=0\}$ at $t=0$. 
For any sufficiently small neighborhood  $\mathcal U = J\times U_\nu(0)$ of $M_k$ there exists $N$ with the following property:  for any nonnegative integer $r$ and any $t$ on the arc
%Let $g_t=\mathcal R_h^r f_t$ and $\rot g_t = \alpha_{N}$. Then for 
\begin{equation}
\label{eq-arc}
J_r = \{t\in I \mid \rot f_t \in [\frac{p_{r+N+1}}{q_{r+N+1}}, \frac{p_{r+N}}{q_{r+N}}]\}
\end{equation}
we have $\mathcal R^j f_t\in \mathcal U$ for all $j=0, 1, \dots, [r/(m+1)]$.

\end{lemma}

\begin{proof}

Choose $\delta>0$ such that  for $t$ in a small neighborhood of zero,  we have $s(\frac{df_t}{dt})<1/\delta$; we will choose $N$ sufficiently large so that the arc $\{f_t\}_{t\in J_r}$ %\eqref{eq-arc}
 is in this neighborhood, for every $r\ge 0$.  

Fix $\eps>0$ with $\mu+\eps<1$ and apply Lemma \ref{lem-inclination} to $\eps, \delta$ to fix a neighborhood $\mathcal U =  J\times U_\nu(0)$ of the set $M_k$.  We assume that the neighborhood $\mathcal U$ is sufficiently small, so that $s(\frac{df_t}{dt})<1/\delta$, whenever $f_t\in\mathcal U$.

Note that $\alpha_j$ is of type bounded by $k$ for all $j\ge 0$, and thus $R_{\alpha_j}\in M_k$. 
 Let $\alpha_j\sim \frac{p^j_{i}}{q^j_{i}}$ be the continued fractional convergents for $\alpha_j$; denote $I_j^k = [p^j_k/q^j_k, p^j_{k+1}/q^j_{k+1}] $. Note that $\alpha_j\in I_j^k$ and $G(I_j^k) = I_{j+1}^{k-1}$.  Choose $N$, $0<\kappa<\nu$ such  that both $\alpha_j$ and the intervals $I_j^N$  do not  approach closer to $\partial J$ than $2\kappa$.  This is possible because the orbit of $\alpha$ cannot accumulate to $\partial J$ (it stays in the compact set  $M_k$ that belongs to $J$) and the intervals $I_j^N$ are uniformly short for large $N$.

We choose $s$ so that $ C_2(\mu+\eps)^{s}<\kappa$. Since $f_0$ and its images under $\mathcal R$ are rotations, by increasing $N$, we can guarantee that the arc $J_r$ is short enough for any $r\ge 0$, so that for any $t\in J_r$, the maps $\mathcal R^j(f_t), j=0, 1, \dots, s$ are $\kappa$-close to rotations and belong to $\mathcal U$. 

Next, we fix any $r\ge 0$, and we prove by induction on $j$ that for any $t\in J_r$, each map $\mathcal R^j(f_t), j=s+1, \dots, [r/(m+1)]$ also belongs to $\mathcal U$ and is $\kappa$-close to rotations in the sense that $\|P_s \mathcal R^j(f_t)\|<\kappa$. The case $j=s$ will serve as an inductive base.

Indeed, suppose that this statement holds for $s, s+1, \dots, j$. Then for $j+1$,  the map $g_t= \mathcal R^{j+1}(f_t)$ is $C_2(\mu+\eps)^{j+1}$--close to rotation due to Lemma \ref{lem-inclination}. This is smaller than $\kappa$ due to the choice of $l$. Thus $\|P_s g_t\|<\kappa<\nu$.

It remains to prove that $P_u g_t\in J$. 
 Let $\mathcal R^{j+1} R_\alpha=\alpha_{i}$; note that $i\le (m+1)(j+1)$ since $\mathcal R$ acts on rotation numbers as a $(m(f)+1)$-st power of the Gauss map $G$ and $m(f)\le m$. Since $(m+1)(j+1)\le (m+1)\cdot [r/(m+1)]\le r$, we have $i\le r$. 
 
 Now, for $f_t$ with $t\in J_r$,
% On the arc \eqref{eq-arc}, 
we have $\rot f_t \in I_0^{r+N}$, thus  $\rot (g_t) \in G^i(I_0^{r+N})= I_i^{r+N-i}$. We have $I_i^{r+N-i}\subset I_i^N$ since $i \le r$. Since $\|P_s g_t\|<\kappa$, the integral $\int_{\bbR/\bbZ} (g_t(z)-z) dz$  belongs to the $2\kappa$-neighborhood of the interval $I_i^N$ %$\rot g_t \in I_i^{N}$ 
and thus belongs to $J$ due to the choice of $N, \kappa$. 
 
 Therefore $P_u g_t \in J$ and $\|P_s g_t\|<\nu$, thus $g_t\in \mathcal U$.

This implies the statement. 
\end{proof}

\textbf{Step 6: End of the proof}

As before, we assume that $t_0=0$ in the statement of Theorem \ref{th-main}.
Assume that after passing to the linearizing chart for $f_0$, the family $\mathcal F=\{f_t\}$ is contained in $\mathcal D_{h_1}^{\bbR}$ for some $h_1>0$, and $f_0=R_\alpha$. Then, according to Step~1, there exists an integer $s>0$ that depends on $\mathcal F$ and $\alpha$, such that the renormalization operator $\mathcal R_{h_1,h,s}$ takes the family $\mathcal F$ to a family $\tilde{\mathcal F}=\{\tilde f_t\}_{t\in I}\subset \mathcal D_h^{\bbR}$.

Choose $\delta>0$ such that  for $t$ in a small neighborhood of zero,  we have $s(\frac{d\tilde f_t}{dt})<1/\delta$. Fix $\eps>0$ and set $\tilde \mu = \tau^l/\lambda^l +\eps$. Find a neighborhood $\mathcal U$ using Lemmas \ref{lem-arc} and \ref{lem-inclination}, applied to the family $\tilde{\mathcal F}=\{\tilde f_t\}$.

Lemmas \ref{lem-arc} and \ref{lem-inclination} imply that for $N$ that depends  on the family $\tilde{\mathcal F}=\{\tilde f_t\}\subset\mathrm D_h^{\bbR}$ (hence, on the family $\mathcal F$), the corresponding arcs 
\begin{equation*}
	J_r = \{t\in I \mid \rot \tilde f_t \in [\frac{p_{r+N+1}}{q_{r+N+1}}, \frac{p_{r+N}}{q_{r+N}}]\}
\end{equation*}
satisfy the following: for any $r\ge 0$, $t\in J_r$ and $j=[[r/(m+1)]/l]$,
 $$
\dist_{C_0(\Pi_{h})} (\mathcal R^{lj} \tilde f_t, rotations) < C_2 \cdot {\tilde \mu}^{j},
$$
where $C_2>0$ depends on $\delta$, hence on the family $\mathcal F$.

Due to Cauchy estimates we have
 $$
\dist_{C^2(\Pi_{h/2})} (\mathcal R^{lj} \tilde f_t, rotations) < C_3 \cdot \tilde \mu^{j}.
$$
Hence the distortion of the maps $g_t=\mathcal R^{lj} \tilde f_t$ is estimated above by $3 C_3 \cdot  \tilde \mu^j$. Theorem \ref{bubble_theorem} implies that the bubble of the renormalized family $g_t$ attached at $p/q\in \bbQ$  will fit inside the disc of radius  $C_3 \cdot {\tilde \mu}^{j}/q^2$ that is tangent to $\bbR$ at $p/q$. 

Theorem \ref{th-renorm-bubbles} and Lemma \ref{lem-PSL} imply that for any bubble of the initial family attached at any point $p/q$ on the segment
$[p_{r+s+N+1}/q_{r+s+N+1}, p_{r+s+N}/q_{r+s+N}]$ fits inside  the disc of radius  $C_3 \cdot {\tilde \mu}^{j}/q^2$ that is tangent to $\bbR$ at $p/q$, where $j = [[r/(m+1)]/l]$. 

This implies the statement of Theorem \ref{th-main} with $\Lambda = \tilde \mu^{1/((m+1)l)}$ and $c(\alpha, \mathcal F) = C_3 \cdot (\tilde \mu)^{-N-s}$.

\end{proof}

The resulting scaling factor is at most  $(\frac{\tau^l}{\lambda^l}+\eps)^{1/((m+1)l)}$ where $\tau, \lambda$ are the expansion and the contraction rate associated with the renormalization operator $\mathcal R$ that corresponds to the ${m(f)}+1$-power of the Gauss map and $m\ge m(f)$. So for periodic orbits of $\mathcal R$, the scaling factor can be  estimated from above by the fraction of top multipliers of $\mathcal R$: the unstable and the top stable multiplier. This implies the following proposition. 

\begin{proposition}
\label{prop-golden}
For the golden ratio rotation number $\alpha=\phi=\frac{\sqrt{5}-1}{2}$, Theorem \ref{th-main} holds with some $\Lambda<\phi^2$.
\end{proposition}
\begin{proof}
Note that $R_\phi$ is a fixed point of the renormalization operator, and $\lambda = \|d|_{R_\phi}\mathcal R 1 \|= |dG^{m+1 }|   = (\phi^{-2})^{m+1}$. Since $\tau<1$, for small $\eps$ we have  $$\Lambda  = (\frac{\tau^l}{\lambda^l}+\eps)^{1/(l(m+1))} < \frac{1}{\lambda^{1/(m+1)}} = \phi^{2}. $$
\end{proof}

\printbibliography

@online{Kapiamba,
author={A. Kapiamba},
title={An optimal Yoccoz inequality for near-parabolic quadratic polynomials},
eprint = {2103.03211},
	eprinttype = {arxiv},
	year = 2025,
	langid = {english}
}

@article{GY,
title = {Analytic linearization of conformal maps of the annulus},
journal = {Advances in Mathematics},
volume = {409},
pages = {108636},
year = {2022},
issn = {0001-8708},
doi = {https://doi.org/10.1016/j.aim.2022.108636},
url = {https://www.sciencedirect.com/science/article/pii/S0001870822004534},
author = {Nataliya Goncharuk and Michael Yampolsky}
}

@article{NGselfsim,
	doi = {10.1088/1361-6544/ab1b8f},
	url = {https://doi.org/10.1088/1361-6544/ab1b8f},
	year = 2019,
	month = {6},
	publisher = {{IOP} Publishing},
	volume = {32},
	number = {7},
	pages = {2496--2521},
	author = {Nataliya Goncharuk},
	title = {Self-similarity of bubbles},
	journal = {Nonlinearity}
}

@article{H,
	author = {J. H. Hubbard},
	title = {Local connectivity of Julia sets and bifurcation loci: three theorems of J.\,C.\,Yoccoz},
	journal = {Topological Methods in Modern Mathematics},
	publisher = {Goldberg and Phillips eds Publish or Perish},
	year = 1993,
	pages = {467—511},
	langid = {english}
}

@article{NG2012-en,
	author = {Nataliya Goncharuk},
	title = {Rotation numbers and moduli of elliptic curves},
	journaltitle = {Functional Analysis and Its Applications},
	volume = 46,
	number = 1,
	pages = {11-25},
	year = 2012,
	langid = {english}
}

@article{XB_NG,
	author = {Xavier Buff and Nataliya Goncharuk},
	title = {Complex rotation numbers},
	journaltitle = {Journal of modern dynamics},
	volume = 9,
	pages = { 169-190},
	year = 2015,
	langid = {english}
}

@article{NGinters,
    author = {Nataliya Goncharuk},
    title = {Complex rotation numbers: bubbles and their intersections},
    journaltitle = {Analysis and PDE},
    volume = 11,
    year = 2018,
    pages = {1787-1801},
    number = 7,
    doi = {10.2140/apde.2018.11.1787},
    langid = {english}
}

@article{YuIM,
	author = {Yulij Ilyashenko and Vadim Moldavskis},
	title = {Morse-Smale circle diffeomorphisms and moduli of complex tori},
	journal = {Moscow Mathematical Journal},
	volume = 3,
	issue = {April-June},
	year = 2003,
	number = 2,
	pages = {531--540},
	langid = {english},
}

@article{M,
	author = {В. С. Молдавский},
	title = {Модули эллиптических кривых и числа вращения диффеоморфизмов окружности},
	journaltitle = {Функц. анализ и его прил.},
	volume = 35,
	number = 3,
	year = 2001,
	pages = {88–91},
	langid = {russian}
}

@article{M-en,
	author = {Vadim Moldavskij},
	title = {Moduli of Elliptic Curves and Rotation Numbers of Circle Diffeomorphisms},
	journaltitle = {Functional Analysis and Its Applications},
	volume = 35,
	number = 3,
	year =  2001,
	pages = {234–236},
	langid = {english}
}

@article{Ris,
	author = {E. Risler},
	title = {Linéarisation des perturbations holomorphes des rotations et applications},
	journaltitle = {Mémoires de la S.M.F.},
	series = 2,
	volume = 77,
	year = 1999,
	pages = {1--102},
	langid = {french}
}

@article{Yoc,
	author = {J.-C. Yoccoz},
	title = {Conjugaison différentiable des difféomorphismes du cercle dont le nombre de rotation vérifie une condition diophantienne},
	journaltitle = {Annales Scientifiques de l’École Normale Superieure},
	shortjournal = {Ann. Sci. École Norm. Sup.},
	series = 4,
	issue = 17,
	year = 1984,
	number = 3,
	pages = {333--359},
	langid = {french}
}

@book{Arn-english,
	author = {Vladimir Igorevich Arnold},
	title = {Geometrical Methods In The Theory Of Ordinary Differential Equations},
	location = {New York -- Berlin},
	series = {Grundlehren der mathematischen Wissenschaften [Fundamental Principles of Mathematical Science]},
	publisher = {Springer-Verlag},
	volume = 250,
	year = 1983,
	pagetotal = 334,
	langid = {english}
}

@article{Walter,
author = {H.-O. Walter}, 
title={Inclination lemmas with dominated convergence},
journal = {Z. angew. Math. Phys.},
volume = 38, 
pages = {327–337},
year =  1987,
doi = {https://doi.org/10.1007/BF00945417}
}

@article{Palis,
title = {A note on the inclination Lemma ($\lambda$-Lemma) and Feigenbaum's rate of approach},
author = {J. Palis},
year = 1983,
journal={Palis, J. (eds) Geometric Dynamics},
series = {Lecture Notes in Mathematics, vol. 1007},
publisher = {Springer, Berlin, Heidelberg} 
}

\end{document}